\documentclass[11pt]{amsart}
\usepackage{eurosym}
\usepackage{amsfonts}
\usepackage{amsmath}
\usepackage{amssymb}
\usepackage{tikz-cd}
\usepackage{bbm}
\usepackage{amscd}
\usepackage{etoolbox}
\usepackage{filecontents}

\makeatletter
\patchcmd{\BR@backref}{\newblock}{\newblock(page~}{}{}
\patchcmd{\BR@backref}{\par}{)\par}{}{}
\makeatother
\newtheorem{theorem}{Theorem}[section]
\newtheorem*{theorem*}{Theorem}

\newtheorem{corollary}[theorem]{Corollary}

\newtheorem{lemma}[theorem]{Lemma}

\newtheorem{proposition}[theorem]{Proposition}

\theoremstyle{definition}
\newtheorem{example}[theorem]{Example}
\newtheorem{definition}[theorem]{Definition}

\AtBeginDocument{   \def\MR#1{}
}

\newcommand\sxa{\mathop{\cl S_{X,A}}}
\newcommand\sxamin{\mathop{\cl S^{\min}_{X,A}}}
\newcommand\syb{\mathop{\cl S_{Y,B}}}
\newcommand\sybmin{\mathop{\cl S^{\min}_{Y,B}}}

\newcommand\jlambda{\mathop{J(\lambda)}}

\newcommand\Ref{\mathop{\rm Ref}}

\newcommand\deter{\mathop{\rm det}}
\newcommand\loc{\mathop{\rm loc}}
\newcommand\qq{\mathop{\rm q}}
\newcommand\qa{\mathop{\rm qa}}
\newcommand\qs{\mathop{\rm qs}}
\newcommand\qc{\mathop{\rm qc}}
\newcommand\ns{\mathop{\rm ns}}
\newcommand\xx{\mathop{\rm x}}

\newcommand\omin{\mathop{\rm omin}}
\newcommand\cc{\mathop{\rm c}}
\newcommand\dd{\mathop{d}}

\newcommand{\cl}[1]{\mathcal{#1}}
\newcommand{\bb}[1]{\mathbb{#1}}

\begin{document}
\title[Perfect strategies for non-signalling games]
{Perfect strategies for non-signalling games} 


\author[M. Lupini et al]{M. Lupini}
\address{Mathematics Department,
California Institute of Technology,
1200 East California Boulevard,
Mail Code 253-37,
Pasadena, CA 91125}
\email{lupini@caltech.edu}

\author[]{L. Man\v{c}inska}
\address{Department of Mathematical Sciences,
Universitetsparken 5, 2100 Copenhagen, Denmark}
\email{mancinska@math.ku.dk}

\author[]{V. I. Paulsen}
\address{Institute for Quantum Computing and Department of Pure Mathematics, University of Waterloo, Waterloo, Canada}
\email{vpaulsen@uwaterloo.ca}

\author[]{D. E. Roberson}
\address{Department of Physics,
Technical University of Denmark,
                    Richard Petersens Plads,
                    Copenhagen, Denmark}
\email{davideroberson@gmail.com}

\author[]{G. Scarpa}
\address{Faculty of Mathematical Sciences,
Universidad Complutense de Madrid,
Plaza de las Ciencias, 3, Dpcho. 489,
28040 Madrid, Spain}
\email{giannicolascarpa@gmail.com}

\author[]{S. Severini}
\address{Department of Computer Science, University College London,
Gower Street, London WC1E 6BT, United Kingdom, 
and 
Institute of Natural Sciences, Shanghai Jiao Tong University, 200240 Shanghai, China}
\email{s.severini@ucl.ac.uk}

\author[]{I. G. Todorov}
\address{Mathematical Sciences Research Centre,
Queen's University Belfast, Belfast BT7 1NN, United Kingdom,
and
School of Mathematical Sciences, Nankai University, 300071 Tianjin, China}
\email{i.todorov@qub.ac.uk}

\author[]{A. Winter}
\address{ICREA and F\'{\i}sica Te\`{o}rica: Informaci\'{o} i Fenomens Qu\`{a}ntics, Universitat Aut\`{o}noma de Barcelona, ES-08193 Bellaterra, Barcelona, Spain}
\email{andreas.winter@uab.cat}

\date{9 April 2018}

\begin{abstract} 
We unify and consolidate various results about 
non-signall-ing games, a subclass of non-local two-player one-round games, 
by introducing and studying several new families of games
and establishing general theorems about them, which extend a number of known facts in a variety of special cases. 
Among these families are {\it reflexive games,} 
which are characterised as the hardest non-signalling games that can be won using a given set of strategies. 
We introduce {\it imitation games,} in which the players 
display linked behaviour, and which contains as subclasses the classes of variable assignment games, 
binary constraint system games, synchronous games, many games based on graphs, and {\it unique} games. 
We associate a C*-algebra $C^*(\mathcal{G})$ to any imitation game $\mathcal{G}$, and show that 
the existence of perfect quantum commuting (resp.\ quantum, local) 
strategies of $\mathcal{G}$ can be characterised in terms of properties of this C*-algebra, extending known results about synchronous games. 
We single out a subclass of imitation games, which we call {\it mirror games,} 
and provide a characterisation of their quantum commuting strategies 
that has an algebraic flavour, showing in addition that their approximately quantum perfect strategies 
arise from amenable traces on the encoding C*-algebra. 
We describe the main classes of non-signalling correlations in terms of states on operator system tensor products.
\end{abstract}

\maketitle




\section{Introduction}

The study of correlations between spatially separated and non-signalling parties has been 
central for Quantum Mechanics since the initiation of the subject. The celebrated Bell Theorem
demonstrates that the set $\cl C_{\rm q}$ of quantum correlations, arising from shared finite dimensional entanglement,
is strictly larger than the set of classical (or local) correlations, arising from shared randomness. 
A third natural class of correlations, $\cl C_{\rm qc}$, arising from Algebraic Quantum Field Theory, 
corresponds to the \emph{commuting model} of Quantum Mechanics.
According to it, the measurements of the two parties act on a single (infinite dimensional) Hilbert space;  
this setting was studied in \cite{jnppsw} and subsequently in \cite{oz}, where the author showed that every 
correlation from $\cl C_{\rm qc}$
can be approximated by ones from $\cl C_{\rm q}$ if and only if the Connes Embedding Problem in operator algebra theory 
\cite{c} has an affirmative answer.
Deep results about the inequality between those classes of correlations,
answering questions left open by Tsirelson (see \cite{tsirelson1980} and \cite{tsirelson1993}) 
were recently obtained by Slofstra in \cite{s_2016} and  \cite{s_2017} when the number of inputs is large and in \cite{dpp} similar results are shown for a small number of inputs.
The relevance of operator algebraic techniques in the study of correlation sets became 
also apparent through \cite{fkpt} and, subsequently, \cite{paulsen_quantum_2015}, 
where operator systems and their tensor products were used to describe some correlation classes.

Non-signalling games form a subclass of the class of non-local, or two-player one-round games, 
and have attracted substantial attention in theoretical physics, mathematics and computer science
(see e.g.\ \cite{mv}, \cite{rao}, \cite{pv} and \cite{sv}).
Non-local correlations have been successfully used to obtain strategies for such games that outperform the classical ones. 
A prominent such example is the graph colouring game defined in \cite{cameron_quantum_2007}, where 
it was demonstrated that the quantum chromatic number of a graph, arising from the set $\cl C_{\rm q}$, 
can be strictly smaller than its classical counterpart. 
In \cite{paulsen_quantum_2015}, the authors defined a corresponding commuting version of the chromatic number, 
using the set $\cl C_{\rm qc}$ of quantum commuting correlations,
in addition to other related analogues. 
The graph homomorphism game was introduced in \cite{qhoms, Rthesis} and subsequently studied \cite{ortiz_quantum_2016} as 
a generalisation of the graph colouring game. An even more general class -- that of synchronous games -- 
was considered in \cite{paulsen_estimating_2016}, where perfect strategies from the classes $\cl C_{\rm q}$ and 
$\cl C_{\rm qc}$ were described in terms of traces on a canonical C*-algebra, associated with the input-output sets of the game.

The present paper is a contribution to this area and aims at utilising an operator algebraic approach 
in order to describe various correlation classes and to 
formulate, in several distinct settings, necessary and sufficient conditions for the existence of a non-signalling correlation from a given correlation class,  that is 
a perfect strategy of a given non-signalling game. 
After collecting some necessary preliminary material in Section \ref{s_prel}, we describe, in Section \ref{Section:correlations},
the perfect strategies of a general non-signalling game that belong to a given class, in terms of states of 
operator system tensor products. 
These descriptions have two advantages as compared to the existing use of C*-algebras, encountered more 
commonly in the literature: first, the operator systems involved are finite dimensional and hence 
easier to handle than their infinite dimensional C*-algebraic counterparts
and, second, they open a way for the study of generalised probabilistic theories, 
not allowed by the C*-algebraic formalism \cite{b}. 

We further define the new class of \emph{reflexive games}; a reflexive game 
can be thought of as the hardest game that can be won 
using a family of strategies from a given class. We show that the perfect strategies of reflexive games 
are in a one-to-one correspondence to states on certain operator system quotients. 
In fact, we introduce, more generally, \emph{reflexive covers} of games, and exhibit several examples where the reflexive 
cover of a game can be strictly harder than the original game. 

In Sections \ref{Section:games} and \ref{s_wsig}, 
we introduce and study the class of \emph{imitation games}. It includes  
a number of classes of games that have been extensively studied previously, 
such as the class of variable assignment games and its 
subclass of binary constraint system (BCS) games \cite{cleve_characterization_2014}, unique games \cite{rao}
and synchronous games \cite{paulsen_estimating_2016}. 
With every imitation game $\cl G$, we associate a 
canonical C*-algebra $C^*(\cl G)$, and provide an explicit description of $C^*(\cl G)$ in the case 
$\cl G$ is a variable assignment game. If $\cl G$ is a linear BCS game, 
we relate $C^*(\cl G)$ to the group C*-algebra of the group of $\cl G$ introduced in \cite{cleve_perfect_2016}. 
We show that the perfect quantum commuting strategies of an imitation game $\cl G$
correspond to traces on $C^*(\cl G)$, while the perfect quantum strategies of $\cl G$ correspond to finite dimensional 
representations of $C^*(\cl G)$. 
We prove the equality of several classes of correlations, namely the quantum ones, the 
quantum spatial ones and the maximally entangled quantum ones.

In Section \ref{s_mg}, we consider a subclass of the class of imitation games, which we call \emph{mirror games},
and provide an algebraic, Hilbert-space free, approach, to their perfect strategies. 
As a result, we give a different representation of their perfect quantum commuting strategies,
using traces on canonical input-output C*-algebras. We show that the
\emph{quantum approximate} perfect strategies of these games correspond precisely to amenable traces 
on these C*-algebras, extending a recent result from \cite{kps}.


\section{Preliminaries}\label{s_prel}

We fix  finite sets $X$, $Y$, $A$ and $B$.
A collection of scalars 
$$p = \{p(a,b|x,y) : (x,y)\in X\times Y, (a,b)\in A\times B\}$$
is called \emph{non-signalling} if 
\begin{equation}\label{eq_yy'}
\sum_{b\in B} p(a,b|x,y) = \sum_{b\in B} p(a,b|x,y'), \ \ x\in X, y,y'\in Y, a\in A
\end{equation}
and 
\begin{equation}\label{eq_xx'}
\sum_{a\in A} p(a,b|x,y) = \sum_{a\in A} p(a,b|x',y), \ \ x,x'\in X,  y\in Y, b\in B.
\end{equation}
If, in addition, $p(\cdot,\cdot | x,y)$ is a probability distribution for every $(x,y)\in X\times Y$,
then $p$ is called a \emph{non-signalling correlation} on $(X,Y,A,B)$.
The set $\cl C$ (also denoted $\cl C_{\rm ns}$) of all non-signalling correlations is canonically endowed 
with a compact metrisable topology and a convex structure by regarding it as a subset of $\bb R^m$ where $m$ is the cardinality of $X \times Y \times A \times B$.
If $p\in \cl C$, we set $p(a|x)$ (resp.\ $p(b|y)$) to be equal to the sum in (\ref{eq_yy'}) 
(resp.\ (\ref{eq_xx'})), for any choice of $y\in Y$ (resp.\ $x\in X$).

We recall the definition of several sets of non-signalling correlations.
The set $\cl C_{\deter}$ of \emph{deterministic} correlations consists of all correlations 
$p$ for which there exist functions $f : X\to A$ and $g : Y\to B$ such that 
$p(a,b|x,y) = 1$ if and only if $a = f(x)$ and $b = g(y)$. 

The set $\cl C_{\mathrm{loc}}$ of \emph{local correlations }is the convex hull
of correlations of the form $p\left( a,b|x,y\right) =p^{1}\left(
a|x\right) p^{2}\left( b|y\right) $ where $p^{1}:A\times X\rightarrow 
\left[ 0,1\right] $ and $p^{2}:B\times Y\rightarrow \left[ 0,1\right] $
satisfy $\sum_{a}p^{1}\left( a|x\right) =\sum_{b}p^{2}\left( b|y\right)
=1$ for every $x\in X$ and every $y\in Y$. 

In order to define the rest of the classes, known as non-classical, 
we recall that a \emph{positive operator-valued
measure (POVM)} on a Hilbert space $\mathcal{H}$ is a tuple $\left(
A_{1},\ldots ,A_{n}\right) $ of positive operators on $\mathcal{H}$ summing
up to the identity operator. A \emph{projection-valued measure (PVM)}  is a
POVM consisting of projection operators. 

In each of the definitions given below, we use PVM's, but results of \cite{fritz_operator_2014} (see also \cite{paulsen_quantum_2015}) show that the sets of correlations that we obtain are the same if we replace PVM's with POVM's in each definition.

The set $\cl C_{\mathrm{q}}$ of 
\emph{quantum correlations} consists of the correlations of the form 
\begin{equation*}
p\left(a,b|x,y\right) =\left\langle \xi | P_{x,a}\otimes Q_{b,y} | \xi
\right\rangle
\end{equation*}
where $\mathcal{H}$ is a \emph{finite-dimensional }Hilbert space, 
$\left\vert \xi \right\rangle \in \mathcal{H}\otimes \mathcal{H}$ is a unit
vector, and for every $x\in X$ and $y\in Y$, $\left( P_{x,a}\right) _{a\in A}$
and $\left(Q_{y,b}\right) _{b\in B}$ are PVM's on $\mathcal{H}$. 


Suppose that $\mathcal{H}$ is
a finite-dimensional Hilbert space with fixed basis $\left\vert
i\right\rangle$, $i=1,2,\ldots ,d$. 
The\emph{\ maximally entangled vector} in 
$\mathcal{H}\otimes \mathcal{H}$ is the unit vector 
$\left\vert \eta \right\rangle = d^{-1/2}\sum_{i=1}^{d}\left\vert i\right\rangle \otimes \left\vert i\right\rangle$. 
The set $\cl C_{\mathrm{qm}}$ is the\emph{\ convex hull} of
the quantum correlations of the form 
\begin{equation*}
p\left( a,b|x,y\right) =\left\langle \eta |P_{x,a}\otimes Q_{b,y}|\eta\right\rangle,
\end{equation*}
where $\left(P_{x,a}\right)_{a\in A}$ (resp.\ $\left(Q_{y,b}\right) _{b\in B}$) 
is a PVM on $\mathcal{H}$ for every $x\in X$ (resp.\ $y\in Y$). 

The set $\cl C_{\mathrm{qs}}$ of \emph{quantum spatial} correlations is defined similarly to the set of 
quantum correlations, but the restriction that $\mathcal{H}$ be finite dimensional is dropped. The 
set $\cl C_{\mathrm{qa}}$ of \emph{quantum approximate}
correlations is defined to be the closure of $\cl C_{\mathrm{q}}$. Finally, the set $\cl C_{\mathrm{qc}}$ of 
\emph{quantum commuting} correlations consists of the correlations of the form
\begin{equation*}
p\left( a,b|x,y\right) =\left\langle \xi |P_{x,a} Q_{y,b}|\xi \right\rangle,
\end{equation*}
where $\mathcal{H}$ is a Hilbert space, 
$\left\vert \xi \right\rangle \in \mathcal{H}$ is a unit vector and 
$\left(P_{x,a}\right) _{a\in A}$ and $\left(Q_{y,b}\right) _{b\in B}$, $x\in X$, $y\in Y$ are
PVMs on $\mathcal{H}$ such that $P_{x,a} Q_{y,b} = Q_{y,b} P_{x,a}$ for all $x\in X$, $y\in Y$, $a\in A$ and $b\in B$.
A separability argument shows
that considering \emph{separable }Hilbert spaces in the definition of
the sets $\cl C_{\mathrm{qs}}$ and $\cl C_{\mathrm{qc}}$ yields equivalent
definitions. It is clear from the definition (after observing that 
$\cl C_{\mathrm{qc}}$ is closed) that we have the following inclusions between
these sets of correlations
\begin{equation*}
\cl C_{\deter} \subseteq \cl C_{\mathrm{loc}}\subseteq \cl C_{\mathrm{qm}}\subseteq \cl C_{\mathrm{q}}\subseteq 
\cl C_{\mathrm{qs}}\subseteq \cl C_{\mathrm{qa}}\subseteq \cl C_{\mathrm{qc}} \subseteq  \cl C_{\rm ns}\text{.}
\end{equation*}

We now recall the connection between non-signalling correlations and perfect strategies for 
non-signalling games. 
A \emph{non-signalling game} is a tuple $\cl G = (X,Y,A,B,\lambda_{\cl G})$, where 
$X,Y,A$ and $B$ are finite sets and $\lambda_{\cl G} : X\times Y \times A\times B \to \{0,1\}$ is a function.
We think of $X$ and $Y$ as sets of possible \emph{inputs} or \emph{questions} for two players
(Alice and Bob) of a two-party single-round game, and of $A$ and $B$ as 
sets of possible \emph{outputs }or \emph{answers }for Alice and Bob,
respectively. The function $\lambda_{\cl G}$ is called the 
\emph{payoff}, or \emph{rule}, \emph{function} of $\cl G$,
assigning value $1$ to $\left(x,y,a,b\right) $ if $a,b$ are
\emph{acceptable answers} to the pair $(x,y)$ of questions, and $0$ otherwise. 
When there is no risk of confusion, 
we write $\lambda = \lambda_{\mathcal{G}}$. 
Notice that in this subclass of non-local games, 
we do not consider probability distributions on the input sets and we restrict our attention to two players. 
Correspondingly, we are interested in the perfect strategies for these games, which are automatically winning 
strategies for any given probability distribution on the direct product of the input sets of the game.
More precisely, we call a non-signalling correlation $p$ on $(X,Y,A,B)$ 
a \emph{perfect} \emph{strategy} for 
$\mathcal{G}$ if 
$$\lambda \left(x,y,a,b\right) = 0 \ \Longrightarrow \ p\left(a,b|x,y\right) =0.$$
We let
$\cl C\left( \mathcal{G}\right)$ (or $\cl C(\lambda)$, $\cl C_{\rm ns}(\lambda)$) be the set of all such correlations.

More specifically, if $\lambda : X\times Y\times A\times B\rightarrow \left\{ 0,1\right\}$ 
and $p\in \cl C_{\ns}$, write
$$N(\lambda) = \{(x,y,a,b)\in X\times Y \times A\times B : \lambda(x,y,a,b) = 0\}$$
and
$$N(p) = \{(x,y,a,b)\in X\times Y \times A\times B : p(a,b | x,y) = 0\}.$$
Setting 
$$\cl C_{\xx}(\lambda) := \{p\in \cl C_{\xx} : N(\lambda) \subseteq N(p)\}= \cl C_{\xx} \cap \cl C(\cl G),$$
we obtain a corresponding chain 
$$\cl C_{\deter}(\lambda)\subseteq \cl C_{\loc}(\lambda) \subseteq \cl C_{\rm qm}(\lambda) \subseteq 
\cl C_{\qq}(\lambda) \subseteq \cl C_{\qs}(\lambda) 
\subseteq \cl C_{\qa}(\lambda) \subseteq \cl C_{\qc}(\lambda)\subseteq \cl C_{\rm ns}(\lambda).$$
A \emph{perfect $\xx$-strategy} for the game $\cl G$ is an element of $\cl C_{\xx}(\lambda)$.

In this paper, we will arrive at characterisations of such sets of
correlations in terms of states on operator systems. 
We refer the reader to 
\cite{paulsen_completely_2002} for an introduction to the basic notions of
non-commutative functional analysis (see also \cite{effros_operator_2000,pisier_introduction_2003}). 
Let $\cl S$ be an operator system, that is, a subspace of a unital C*-algebra 
$\cl A$ such that $1\in \cl S$ and $x\in \cl S \Rightarrow x^*\in \cl S$. 
Then $M_n(\cl S)\subseteq M_n(\cl A)$, and we let 
$M_n(\cl S)^+$ be the cone of all elements of $M_n(\cl S)$ that are positive in the C*-algebra 
$M_n(\cl A)$. 

Let $\cl S$ and $\cl T$ be operator systems. Given a linear map $\phi : \cl S\to\cl T$, let 
$\phi^{(n)} : M_n(\cl S)\to M_n(\cl T)$ be the map given by $\phi^{(n)}((x_{i,j})_{i,j}) = (\phi(x_{i,j})_{i,j})$. 
The map $\phi$ is called \emph{positive} if $\phi(\cl S^+)\subseteq \cl T^+$, and
\emph{completely positive} if $\phi^{(n)}\left(M_n(\cl S)^+\right)\subseteq M_n(\cl T)^+$
for every $n\in \bb{N}$. 
A state on $\cl S$ is a positive linear map $s : \cl S\to \bb{C}$ with $s(1) = 1$. 

We write $\cl S\subseteq_{\rm c.o.i.} \cl T$ when $\cl S\subseteq \cl T$ 
and $M_n(\cl S)^+ = M_n(\cl T)^+\cap M_n(\cl S)$ for each $n\in \bb{N}$.
We denote by $\cl S\oplus^{1} \cl T$ the coproduct of $\cl S$ and $\cl T$ in the
category of operator systems and unital completely positive maps;
it is characterised by the following universal property: $\cl S\oplus^{1} \cl T$ is generated as a linear space 
by $\cl S$ and $\cl T$, its unit is also the unit of $\cl S$ and of $\cl T$, and 
whenever $\cl R$ is an operator system and $\phi : \cl S\to \cl R$ and $\psi : \cl T\to \cl R$ 
are unital completely positive maps then there exists a unique (unital) completely positive map 
$\theta : \cl S\oplus^{1} \cl T \to \cl R$ extending $\phi$ and $\psi$.
We refer the reader to \cite[Section 3]{fritz_operator_2014} and \cite[Section 8]{kavruk} for further properties of 
the operator system coproduct. 
We further let $\cl S\oplus ^{\infty } \cl T$ be the product of $\cl S$ and $\cl T$, and 
define $\ell ^{\infty}(A)$ to be the product of $\left\vert A\right\vert $ copies of $\mathbb{C}$ indexed by $A$.

In the sequel, we will make use of the tensor theory of operator systems developed in 
\cite{kptt_tensor}. If $\cl S$ and $\cl T$ are operator systems, we denote by 
$\cl S\otimes_{\min}\cl T$ (resp.\ $\cl S\otimes_{\rm c}\cl T$, $\cl S\otimes_{\max}\cl T$)
the minimal (resp.\ commuting, maximal) tensor product of $\cl S$ and $\cl T$
introduced therein. 
We note that, if $\cl S\subseteq_{\rm c.o.i.} \cl A$ and $\cl T\subseteq_{\rm c.o.i.} \cl B$
for C*-algebras $\cl A$ and $\cl B$, then 
$\cl S\otimes_{\min}\cl T\subseteq_{\rm c.o.i.} \cl A\otimes_{\min}\cl B$, where 
$\cl A\otimes_{\min} \cl B$ is the spatial tensor product of $\cl A$ and $\cl B$. 
By their definition, the tensor product $\cl S\otimes_{\rm c} \cl T$ 
linearises pairs of unital completely positive maps $\phi : \cl S\to \cl A$ 
and $\psi : \cl T\to \cl A$ with commuting ranges, 
while the maximal tensor product $\cl S\otimes_{\max} \cl T$ linearises 
jointly completely positive maps $\theta : \cl S\times \cl T\to \cl A$
(here $\cl A$ is an arbitrary C*-algebra).


\section{Correlations as perfect strategies\label{Section:correlations}}

Let $X$, $Y$, $A$ and $B$ be finite sets. 
Following \cite{paulsen_quantum_2015}, we let $\mathcal{S}_{X,A}$ be the
coproduct of $\left\vert X\right\vert $ copies of $\ell^{\infty}(A)$, indexed by $X$. 
Let 
$\cl A(X,A) = \ell^{\infty}(A) \ast_1 \cdots \ast_1 \ell^{\infty}(A)$ be the C*-algebra free product, 
amalgamated over the unit, of $|X|$ copies of $\ell ^{\infty}(A)$; 
note that, via Fourier transform,
$\cl A(X,A) \cong C^*(\bb{F}(A,X))$, where $\bb{F}(X,A) = \bb{Z}_{|A|} \ast \cdots \ast \bb{Z}_{|A|}$ is the free product of 
$|X|$ copies of the cyclic group with $|A|$ elements. 
Letting $(e_{x,a})_{a=1}^{|A|}$ be the canonical basis of the $x$-th copy of $\ell^{\infty}(A)$,
we have that
$$\sxa = {\rm span}\{e_{x,a} : x\in X, a\in A\}$$
within $\cl A(X,A)$. 
Note the relations 
$$\sum_{a\in A} e_{x,a} = 1, \ \ \ x\in X.$$

Set 
$\cl A_{\min}(X,A) = \ell^{\infty}(A)\otimes\cdots \otimes \ell^{\infty}(A)$ ($|X|$ copes)
and note that the C*-algebra $\cl A_{\min}(X,A)$ is *-isomorphic to $\ell^{\infty}(\Delta_{X,A})$,
where $\Delta_{X,A} = A^X$. 
Let 
$$\sxamin = {\rm span}\{e'_{x,a} : x\in X, a\in A\}\subseteq \cl A_{\min}(X,A),$$
where $e'_{x,a}(x',a') = 0$ if $x = x'$ and $a\neq a'$, and  $e'_{x,a}(x',a') = 1$ otherwise.

For conceptual convenience, we will denote the canonical generators of $\cl S_{Y,B}$
(resp.\ $\sybmin$) by $f_{y,b}$ (resp.\ $f_{y,b}'$).
For an element $s$ of the dual vector space 
$(\sxa\otimes \syb)^{\dd}$ of $\sxa\otimes \syb$, write 
$$p_s(a,b|x,y) = s(e_{x,a}\otimes f_{y,b}), \ \ \ (x,y)\in X\times Y, (a,b)\in A\times B,$$
and 
$$p_s = \{p_s(a,b|x,y) : (x,y)\in X\times Y, (a,b)\in A\times B\}.$$ 
Clearly, the collection $p_s$ is non-signalling. Conversely, given a non-signalling collection of scalars
$p$, let $s_p \in (\sxa\otimes \syb)^{\dd}$ be the (well-defined and unique linear) functional satisfying 
$$s_p(e_{x,a}\otimes f_{y,b}) = p(a,b|x,y), \ \ \ (x,y)\in X\times Y, (a,b)\in A\times B.$$
It is clear that $p\to s_p$ is a 
bijective correspondence between 
$(\sxa\otimes \syb)^{\dd}$ and the set of all non-signalling collections on $(X,Y,A,B)$.

Let $\lambda : X\times Y \times A\times B\to \{0,1\}$. 
Set 
$$\jlambda = {\rm span}\{e_{x,a}\otimes f_{y,b} : \lambda(x,y,a,b) = 0\};$$
thus, $\jlambda$ is a linear subspace of $\sxa \otimes \syb$. 

If $\tau$ is any of the tensor products $\max, \cc$ or $\min$ and $J\subseteq \sxa\otimes\syb$, let 
$$\cl P_{\tau}(J) = \{s\in (\sxa\otimes_{\tau}\syb)^{\dd} : s \mbox{ is a state with } J\subseteq \ker(s)\}.$$
Let also 
$$\cl P_{\rm omin}(J) = \{s\in (\sxamin\otimes_{\min}\sybmin)^{\dd} : s \mbox{ is a state with } J\subseteq \ker(s)\}.$$
We write $\cl P_{\tau} = \cl P_{\tau}(\{0\})$.

\begin{theorem}\label{th_cp}
Let $\cl G = (X,Y,A,B,\lambda)$ be a non-signalling game and 
$$p = \{(p(a,b|x,y)) : (x,y)\in X\times Y, (a,b)\in A\times B\}$$
be a non-signalling collection of scalars. 
The map $p\to s_p$ is a continuous affine isomorphism between 
\begin{itemize}
\item[(i)] $\cl C_{\rm ns}(\lambda)$ and $\cl P_{\max}(\jlambda)$;

\item[(ii)] $\cl C_{\qc}(\lambda)$ and $\cl P_{\cc}(\jlambda)$;

\item[(iii)] $\cl C_{\qa}(\lambda)$ and $\cl P_{\min}(\jlambda)$;

\item[(iv)] $\cl C_{\loc}(\lambda)$ and $\cl P_{\rm omin}(\jlambda)$,
\end{itemize}
\noindent and a bijection between
\begin{itemize}
\item[(v)] $\cl C_{\deter}(\lambda)$ and the extreme points of the set $\cl P_{\rm omin}(\jlambda)$. 
\end{itemize}
\end{theorem}
\begin{proof}
(i) 
Set 
$$\cl R_{X,A} = \left\{(z_{x,a})_{x\in X, a\in A} : 
\sum_{a\in A} z_{x,a} = \sum_{a\in A} z_{x',a}, \mbox{ for all } x,x'\in X\right\},$$
viewed as an operator subsystem of $\ell^{\infty}(X\times A)$. 
By \cite[Theorem 5.9]{fkpt_NYJ}, the dual operator system 
$\cl R_{X,A}^{\dd}$ of $\cl R_{X,A}$, which is again an operator system by finite dimensionality \cite[Theorem 4.4]{CE2}, 
is completely order isomorphic to $\sxa$, the duality being given by 
$$\langle (z_{x,a}), e_{x',a'}\rangle = z_{x',a'}.$$
By \cite[Proposition 6.2]{kptt2010} and the fact that $\sxa$ is finite-dimensional, $\mathcal{S}_{X,A}^{d} \cong \cl R_{X,A}$. 
By \cite[Propositon 1.9]{fp}, 
$$(\sxa\otimes_{\max}\syb)^{\dd} \cong \cl R_{X,A}\otimes_{\min}\cl R_{X,A}.$$
By the injectivity of the minimal tensor product, 
$\cl R_{X,A}\otimes_{\min}\cl R_{X,A}$ is (completely order isomorphic to) an operator subsystem of 
$\ell^{\infty}(X\times A \times Y \times B)$, and it is straightforward to check that the image 
of the state space of $\sxa\otimes_{\max}\syb$ under this isomorphism is precisely the set of all non-signalling correlations. 
Statement (i) is now clear. 

(ii) 
By \cite[Lemma 2.6]{paulsen_quantum_2015}, 
$\sxa\otimes_{\cc}\syb\subseteq_{\rm c.o.i.} \cl A(X,A)\otimes_{\max}\cl A(Y,B)$
and hence every state on $\sxa\otimes_{\cc}\syb$ extends to a state on $\cl A(X,A)\otimes_{\max}\cl A(Y,B)$.
The proof now follows the arguments in \cite[Theorem 2.8]{paulsen_quantum_2015}.

(iii) Here the proof follows the one of \cite[Theorem 2.9]{paulsen_quantum_2015}, using the fact that 
$\sxa\otimes_{\min}\syb\subseteq_{\rm c.o.i.} \cl A(X,A)\otimes_{\min}\cl A(Y,B)$. 

(iv) Note that 
\begin{equation}\label{eq_delta}
\sxamin\otimes_{\min}\sybmin \subseteq_{\rm c.o.i.} \ell^{\infty}(\Delta_{X,A}\times\Delta_{Y,B}).
\end{equation}
Suppose that $p_x^1 = (p^1(a|x))_{a\in A}$ is a probability distribution on $A$; 
then $p_x^1$ gives rise to a state $s_x^1$ on $\ell^{\infty}(A)$, $x\in X$. 
The product state $s^1 = \otimes_{x\in X} s_x^1$ on $\ell^{\infty}(X\times A)$ 
is given by 
$s_1\left(\sum_{x,a} \lambda_{x,a} e_{x,a}'\right) = \sum_{x,a} p^1(a|x) \lambda_{x,a}$. 
If, similarly, $p_y^2 = (p^2(b|y))_{b\in B}$ is a probability distribution on $B$, $y\in Y$, 
and $s^2$ is the associated state on $\ell^{\infty}(Y\times B)$, then 
the product state $s^1\otimes s^2$ on $\ell^{\infty}(X\times A\times Y \times B)$
arises from the families of product distributions $p^1_x\otimes p^2_y$, $x\in X$, $y\in Y$. 
It now follows that, if $p\in \cl C_{\loc}$ then $s_p$ is a state on $\sxamin\otimes_{\min}\sybmin$. 

Conversely, suppose that $s$ is a state on $\sxamin\otimes_{\min}\sybmin$. In view of (\ref{eq_delta}), 
$s$ has an extension, which we denote in the same way, to a state on $\ell^{\infty}(\Delta_{X,A}\times\Delta_{Y,B})$. 
It is thus a convex combination of pure states. On the other hand, 
if $s_0$ is a pure state of $\ell^{\infty}(\Delta_{X,A}\times\Delta_{Y,B})$ then there exists $(t_1,t_2)\in \Delta_{X,A}\times\Delta_{Y,B}$
such that $s_0(f) = f(t_1,t_2)$, $f\in \ell^{\infty}(\Delta_{X,A}\times\Delta_{Y,B})$. 
Thus, $s_0 = s_1\otimes s_2$, where $s_1$ (resp.\ $s_2$) is the state on $\ell^{\infty}(\Delta_{X,A})$ 
(resp.\ $\ell^{\infty}(\Delta_{Y,B})$) of evaluation on $t_1$ (resp.\ $t_2$), and it follows that
$p_s$ is a local correlation. 
Statement (iv) is now immediate. 

(v) Suppose that $p\in \cl C_{\deter}(\lambda)$. Clearly, $p$ is an extreme point of $\cl C_{\loc}(\lambda)$; 
by (iv), $s_p$ is an extreme point of $\cl P_{\rm omin}(J(\lambda))$. Conversely, suppose that $s_p$ is an extreme point of 
$\cl P_{\rm omin}(J(\lambda))$. 
By (iv), $p$ is an extreme point of $\cl C_{\loc}(\lambda)$. 
Suppose that $p = t p_1 + (1-t) p_2$, where 
$0\leq t \leq 1$ and $p_1, p_2\in \cl C_{\loc}$.
It follows that $p_1(a,b|x,y) = p_2(a,b|x,y) = 0$ whenever $\lambda(x,y,a,b) = 0$, and hence $p = p_1 = p_2$. 
Thus, $p$ is an extreme point of $\cl C_{\loc}$ and therefore belongs to $\cl C_{\deter}$. 
Thus, $p\in \cl C_{\deter}(\lambda)$. 
\end{proof}

We record, in the following two statements, some additional descriptions of 
the sets of quantum commuting and quantum approximate correlations that will be used in the sequel.

\begin{corollary}\label{th_qcgen}
Let $p$ be a non-signalling correlation on $X\times Y\times A\times B$.
The following statements are equivalent:

\begin{enumerate}
\item[(i)] $p\in \cl C_{\mathrm{qc}}$;

\item[(ii)] $p\left( a,b|x,y\right) =s\left( e_{x,a}\otimes f_{y,b}\right) $ for
some state $s$ on 
$\cl A(X,A)\otimes _{\max} \cl A(Y,B)$;

\item[(iii)] $p\left( a,b|x,y\right) =s\left( e_{x,a}\otimes f_{y,b}\right) $ for
some state $s$ on 
$\cl S_{X,A}\otimes _{\rm c} \cl S_{Y,B}$.
\end{enumerate}
\end{corollary}
\begin{proof}
(i)$\Rightarrow$(iii) follows from Theorem \ref{th_cp} (ii). 

(iii)$\Rightarrow$(ii) follows from Krein's Theorem and the fact that 
$\cl S_{X,A}\otimes _{\rm c} \cl S_{Y,B}\subseteq_{\rm c.o.i.} \cl A(X,A)\otimes _{\max} \cl A(Y,B)$ 
\cite[Lemma 2.7]{paulsen_quantum_2015}. 

(ii)$\Rightarrow$(i) is similar to the arguments in the proof of \cite[Lemma 2.7]{paulsen_quantum_2015}. 
\end{proof}

\begin{corollary}\label{th_qagen}
Let $p$ be a non-signalling correlation on $X\times Y\times A\times B$.
The following statements are equivalent:

\begin{enumerate}
\item[(i)] $p\in \cl C_{\mathrm{qa}}$;

\item[(ii)] 
$p\left( a,b|x,y\right) =s(e_{x,a}\otimes f_{y,b}) $ for
some state $s$ on $\cl A(X,A) \otimes _{\min} \cl A(Y,B)$;

\item[(iii)] 
$p\left( a,b|x,y\right) =s\left( e_{x,a}\otimes f_{y,b}\right) $ for
some state $s$ on $\cl S_{X,A} \otimes _{\min} \cl S_{Y,B}$.
\end{enumerate}
\end{corollary}
\begin{proof}
(i)$\Rightarrow$(iii) follows from Theorem \ref{th_cp} (iii). 

(iii)$\Rightarrow$(ii) follows from Krein's Theorem and the fact that, by the injectivity of the 
minimal tensor product,  
$\cl S_{X,A}\otimes _{\min} \cl S_{Y,B}\subseteq_{\rm c.o.i.} \cl A(X,A)\otimes _{\min} \cl A(Y,B)$. 

(ii)$\Rightarrow$(i) follows by the arguments in the proof of 
\cite[Theorem 2.9]{paulsen_quantum_2015}. 
\end{proof}


In this rest of this section, we give a more precise description of perfect strategies for a special class of games, 
which we now introduce. 
If $\cl G_1 = (X,Y,A,B,\lambda_1)$ and $\cl G_2 = (X,Y,A,B,\lambda_2)$ are games, 
we write $\cl G_1\leq \cl G_2$ if $\lambda_1\leq \lambda_2$,
and in this case say that $\cl G_1$ is \emph{harder} (or \emph{smaller}) than $\cl G_2$. 

Let $\xx\in \{\det, \loc, {\rm qm}, {\rm q}, \qs, \qa, \qc, \ns\}$. For $\Sigma\subseteq \cl C_{\xx}$, 
let $\lambda_{\Sigma} : X\times Y\times A \times B\to \{0,1\}$ be the function defined by the equality 
$$N(\lambda_{\Sigma}) = \cap_{p\in \Sigma} N(p).$$
Clearly, $\lambda_{\Sigma}$ is the rule function of the hardest game for which every element of $\Sigma$ is 
a perfect strategy.

Let $\cl G = (X,Y,A,B,\lambda)$ be a game. 
We set $\lambda_{\xx} = \lambda_{\cl C_{\xx}(\lambda)}$. Thus, 
$$\lambda\mbox{}_{\xx}(x,y,a,b) = 0 \ \Longleftrightarrow \ p(a,b|x,y) = 0 \mbox{ for every } p\in \cl C\mbox{}_{\xx}(\lambda).$$
Set $\Ref_{\xx}(\cl G) = (X,Y,A,B,\lambda_{\xx})$ 
and call it the \emph{reflexive $\xx$-cover} of $\cl G$. 
We call $\cl G$ \emph{$\xx$-reflexive} if $\Ref_{\xx}(\cl G) = \cl G$.

Note the inequalities
$$\lambda\mbox{}_{\det}\leq \lambda\mbox{}_{\loc} \leq \lambda\mbox{}_{\rm qm} \leq \lambda\mbox{}_{\qq} 
\leq \lambda\mbox{}_{\qs}
\leq \lambda\mbox{}_{\qa}\leq \lambda\mbox{}_{\qc}\leq \lambda\mbox{}_{\ns}\leq \lambda.$$

\begin{example}\label{ex_detnonr}
{\rm Consider the graph colouring game for the graph $G$ with vertex set $X = \{1,2,3,4\}$ 
and edge set $\{(1,2), (2,3), (3,4)\}$. Then every deterministic $2$-colouring of $G$ is also a deterministic 
$2$-colouring of the $4$-cycle. This shows that the reflexive covers of a game can be strictly harder than the original game.}
\end{example}

\begin{proposition}\label{p_refref}
Let $\cl G = (X,Y,A,B,\lambda)$ be a game and 
$\xx\in \{\det, \loc, {\rm qm}, {\rm q},$ $\qs, \qa, \qc, \ns\}$. The following hold:
\begin{itemize}
\item[(i)] 
$\Ref_{\xx}(\Ref_{\xx}(\cl G)) = \Ref_{\xx}(\cl G)$.
\item[(ii)] $\cl G$ is $\xx$-reflexive if and only if there exists a set $\Sigma\subseteq \cl C_{\xx}$ 
such that $\lambda = \lambda_{\Sigma}$.
\end{itemize}
\end{proposition}
\begin{proof}
(i) Clearly, $N(\lambda)\subseteq N(\lambda_{\xx})$, and so $\cl C_{\xx}(\lambda_{\xx})\subseteq \cl C_{\xx}(\lambda)$. 
Suppose that $p\in  \cl C_{\xx}(\lambda)$ 
and $(x,y,a,b)\in N(\lambda_{\xx})$. By the definition of $\lambda_{\xx}$, we have that 
$p(a,b|x,y) = 0$. 
Thus, $p\in \cl C_{\xx}(\lambda_{\xx})$ and so $\cl C_{\xx}(\lambda) = \cl C_{\xx}(\lambda_{\xx})$. 
It now follows that $\lambda_{\xx} = \lambda_{\xx\xx}$, that is, $\Ref_{\xx}(\cl G) = \Ref_{\xx}(\Ref_{\xx}(\cl G))$.

(ii) If $\cl G$ is reflexive then we can take $\Sigma = \cl C_{\xx}(\lambda)$. 
Conversely, suppose that $\lambda = \lambda_{\Sigma}$ for some $\Sigma\subseteq \cl C_{\xx}$. 
The inclusion $N(\lambda)\subseteq N(\lambda_{\xx})$ follows by the definition of $\lambda_{\xx}$. 
On the other hand, since $\Sigma\subseteq \cl C_{\xx}(\lambda)$, we have
$$\cap\{N(p) : p\in \cl C\mbox{}_{\xx}(\lambda)\}\subseteq \cap\{N(p) : p\in \Sigma\},$$
that is, $N(\lambda_{\xx})\subseteq N(\lambda)$. Thus, $N(\lambda_{\xx}) = N(\lambda)$
and $\cl G$ is $\xx$-reflexive. 
\end{proof}

A \emph{kernel} in an operator system $\cl S$ \cite{kptt2010} is a subspace $J\subseteq \cl S$
for which there exist an operator system $\cl T$ and a completely positive map $\phi : \cl S \to \cl T$ such that 
$J = \ker(\phi)$. 
If $J\subseteq \cl S$ is a kernel, then the quotient linear space $\cl S/J$ can be equipped with 
a (unique) operator system structure with the property that, whenever $\cl T$ is an operator system and $\phi : \cl S\to \cl T$
is a completely positive map with $J\subseteq {\rm ker}(\phi)$, the induced map $\dot{\phi} : \cl S/ J\to \cl T$
is completely positive. 
For a fixed game $\cl G = (X,Y,A,B,\lambda)$ and $\tau \in \{\max, \cc, \min,\omin\}$, write 
$$J\mbox{}_{\tau}(\lambda) = \bigcap\{\ker(s) : s\in \cl P\mbox{}_{\tau}(\jlambda)\}.$$
By \cite[Proposition 3.1]{kptt2010}, $J_{\tau}(\lambda)$ is a kernel in 
$\cl S_{X,A}\otimes_{\tau} \cl S_{Y,B}$ in the case $\tau \in \{\max,{\rm c},\min\}$, and 
in  $\cl S^{\min}_{X,A}\otimes_{\min} \cl S^{\min}_{Y,B}$ in the case $\tau = {\rm omin}$.

\begin{theorem}\label{th_quotient}
Let $\cl G = (X,Y,A,B,\lambda)$ be a game. Then
\begin{equation}\label{eq_refmax}
\lambda_{\ns}(x,y,a,b) = 0 \ \Longleftrightarrow \  e_{x,a}\otimes e_{y,b} \in J_{\max}(\lambda);
\end{equation}
\begin{equation}\label{eq_refmax1}
\lambda_{\qc}(x,y,a,b) = 0 \ \Longleftrightarrow \  e_{x,a}\otimes e_{y,b} \in J_{\cc}(\lambda);
\end{equation}
\begin{equation}\label{eq_refmax2}
\lambda_{\qa}(x,y,a,b) = 0 \ \Longleftrightarrow \  e_{x,a}\otimes e_{y,b} \in J_{\min}(\lambda);
\end{equation}
\begin{equation}\label{eq_refmax3}
\lambda_{\det}(x,y,a,b) = 0 \ \Longleftrightarrow \  e_{x,a}\otimes e_{y,b} \in J_{\omin}(\lambda).
\end{equation}

Moreover, the map $p \to s_p$ defines a one-to-one correspondence between
\begin{itemize}
\item[(i)] 
the perfect non-signalling strategies for $\Ref_{\ns}(\cl G)$ and the states 
on $(\sxa\otimes_{\max}\syb)/J_{\max}(\lambda)$;
\item[(ii)] 
the perfect non-signalling strategies for $\Ref_{\qc}(\cl G)$ and the states 
on $(\sxa\otimes_{\rm c}\syb)/J_{\cc}(\lambda)$;
\item[(iii)] 
the perfect non-signalling strategies for $\Ref_{\qa}(\cl G)$ and the states 
on $(\sxa\otimes_{\min}\syb)/J_{\min}(\lambda)$;
\item[(iv)] 
the perfect non-signalling strategies for $\Ref_{\loc}(\cl G)$ and the states 
on $(\sxamin\otimes_{\min}\sybmin) / J_{\omin}(\lambda)$. Also, $\lambda_{\det} = \lambda_{\loc}$.
\end{itemize}
\end{theorem}

\begin{proof}
The equivalences (\ref{eq_refmax})-(\ref{eq_refmax3}) follow from Theorem \ref{th_cp}. 

(i) We have that $p$ is a perfect non-signalling strategy for $\Ref_{\ns}(\cl G)$ if and only if 
$p$ is a perfect non-signalling strategy for $\cl G$, if and only if 
$s_p$ annihilates $J(\lambda)$, if and only if $s_p$ annihilates 
$J_{\max}(\lambda)$, if and only if $s_p$ induces a state on the quotient operator system 
$(\sxa\otimes_{\max}\syb)/J_{\max}(\lambda)$. 

Now suppose that $p \ne p^{\prime}$ are two perfect non-signaling strategies for $\Ref_{\ns}(\cl G)$, then because their corresponding states are well-defined on the quotient, we have that for some $a,b,x,y$,
\[
s_p( e_{x,a} \otimes f_{y,b} + J_{\max}(\lambda))= p(a,b|x,y)  \ne p^{\prime}(a,b|x,y)= s_{p^{\prime}}(e_{x,a} \otimes f_{y,b} + J_{\max}(\lambda) ),
\]
which shows that the correspondence is one-to-one. The proofs of the correspondence for
(ii), (iii), and (iv) are similar to (i).

Finally, to see that $\lambda_{\det}= \lambda_{\loc}$, let $p\in \cl C_{\loc}$. By Theorem \ref{th_cp}, there exist $l\in \bb{N}$, 
$p_k\in \cl C_{\deter}$, and $t_k\in (0,1]$, $k = 1,\dots, l$, such that $p = \sum_{k=1}^l t_k p_k$. 
It follows that $N(p) = \cap_{k=1}^l N(p_k)$. 
Thus, 
$$\cap \{N(p) : p\in \cl C_{\loc}(\lambda)\} = \cap \{N(p) : p\in \cl C_{\deter}(\lambda)\},$$
and hence $\lambda_{\loc} = \lambda_{\deter}$. 
\end{proof}

In view of Example \ref{ex_detnonr}, it is natural to consider chromatic non-signalling covers of graphs. 
Let $G$ be a graph and $\xx\in \{\det, \loc, {\rm qm}, {\rm q}, \qs, \qa, \qc, \ns\}$. 
As customary, let $\chi_{\xx}(G)$ be the ${\rm x}$-chromatic number of $G$ (see \cite{paulsen_quantum_2015}). 
The \emph{chromatic $\xx$-cover} 
${\rm Chrom}_{\xx}(G)$ of $G$ 
is the largest supergraph that is coloured by every $\chi_{\xx}(G)$-colouring of $G$.
We write ${\rm Chrom}(G) = {\rm Chrom}_{\det}(G)$. 
Example \ref{ex_detnonr} shows that ${\rm Chrom}(G)$ 
can be strictly larger than $\chi(G)$.

\begin{corollary}\label{c_same}
For any graph $G$, we have that ${\rm Chrom}(G) = {\rm Chrom}_{\loc}(G)$.
\end{corollary}

\begin{example}\label{ex_qcnonr}
Let $G$ and $H$ be graphs. 
If $u$ and $v$ are adjacent vertices, we will write $u\sim v$. 
A \emph{graph homomorphism} from $G$ to $H$ is a function $\varphi: V(G) \to V(H)$ that preserves adjacency, i.e., 
such that if $u, v \in V(G)$ and $u\sim v$, then 
$\varphi(u) \sim \varphi(v) \in V(H)$. 
The \emph{$(G,H)$-homomorphism game} (see~\cite{qhoms,Rthesis}) has input sets $X = Y = V(G)$,
output sets $A = B = V(H)$, and rule function $\lambda$ is given by
\[
\lambda(x,y,a,b) = \begin{cases}
0 & \text{if } (x = y \ \& \ a \ne b) \text{ or } (x\sim y \ \& \ a\not\sim b) \\
1 & \text{otherwise.}
\end{cases}
\]
In other words, the two players are given vertices of $G$ and must respond with vertices of $H$. If they both receive the same vertex of $G$, then they must answer with the same vertex of $H$. If they are given adjacent vertices of $G$, they must respond with adjacent vertices of $H$.
It is easy to notice that the deterministic strategies for the $(G,H)$-homomorphism game
are in one-to-one correspondence with the homomorphisms from $G$ to $H$.

A \emph{walk} of length $\ell$ in a graph $G$ is a sequence $u_0,u_1, \ldots, u_\ell$ of vertices of $G$ 
such that $u_{i-1} \sim u_{i}$ for all $i \in \{1, \ldots, \ell\}$. 
We say that this is a walk from $u_0$ to $u_\ell$, and call $u_0$ and $u_\ell$ the endpoints of the walk.
Suppose that $\varphi$ is a homomorphism from $G$ to $H$. 
Clearly, if $u_0, \ldots, u_\ell$ is a walk of length $\ell$ in $G$, then 
$\varphi(u_0), \ldots, \varphi(u_\ell)$ is a walk of length $\ell$ in $H$. 
Therefore, if $u,v \in V(G)$ are the endpoints of a walk of length $\ell$ in $G$, 
then $\varphi(u), \varphi(v) \in V(H)$ are the endpoints of a walk of length $\ell$ in $H$. 
It follows that $\lambda_{{\rm loc}} (x,y,a,b) = 0$ when there is a walk of length $\ell$ 
with endpoints $x,y \in V(G)$, but there is no walk of length $\ell$ with endpoints $a,b \in V(H)$. 
It follows from this that, for some choices of $G$ and $H$, 
the $(G,H)$-homomorphism game is not loc-reflexive. 
In fact, a special case of this is given in Example~\ref{ex_detnonr}. There, the graph $G$ is the path on four vertices, and $H$ is the complete graph on two vertices. In $G$, vertices 1 and 4 have a walk of length 3 between them. In $H$, there is no walk of length 3 between a vertex and itself, and so vertices 1 and 4 must be mapped to distinct vertices of $H$. 

Surprisingly, it was shown in~\cite{qhoms} that a similar fact holds even for quantum strategies for the homomorphism game. More precisely, if $p \in \mathcal{C}_q$, is a perfect correlation for the $(G,H)$-homomorphism game, and $x,y \in V(G)$ are 
the endpoints of a walk of length $\ell$ in $G$, 
then $p(a,b|x,y) = 0$ unless $a,b \in V(H)$ are the endpoints of a walk of length $\ell$ in $H$. 
Therefore, as in the deterministic/local case, we have that $\lambda_{{\rm q}}(x,y,a,b) = 0$ 
whenever there is a walk of length $\ell$ in $G$ with endpoints $x,y \in V(G)$, 
but there is no walk of length $\ell$ in $H$ with endpoints $a,b \in V(H)$. 
This can be used to show that, for certain graphs $G$ and $H$, the $(G,H)$-homomorphism game is not q-reflexive.

Let us consider the special case where $x = y$ and there is a walk of length $\ell$ beginning and ending at $x \in V(G)$ (this is called a \emph{closed} walk). In this case, if $p$ is a perfect q-strategy, we have that $p(a,b|x,x) = 0$ unless $a = b$ (this is by definition) and $a$ is contained in a closed walk of length $\ell$. Furthermore, since $p(a,b|x,x) = 0$ whenever $b \ne a$, we have that,
if $p(a|x)$ is the corresponding marginal probability 
(given by $p(a|x) = \sum_{b \in V(H)}p(a,b|x,x)$) then $p(a|x) = p(a,a|x,x)$. 
Therefore, for any perfect q-strategy $p$ 
of the $(G,H)$-homomorphism game, if $x \in V(G)$ is contained in a closed walk of length $\ell$, then we have that the marginal $p(a|x)$ is equal to zero unless $a \in V(H)$ is also contained in a closed walk of length $\ell$. Thus, if Alice (or Bob) receive a vertex contained in a closed walk of length $\ell$, then they must respond with a vertex contained in a closed walk of length $\ell$ if they are employing a perfect q-strategy. This implies that $\lambda_{{\rm q}}(a,b|x,y) = 0$ if $x \in V(G)$ is contained in a closed walk of length $\ell$, but $a \in V(H)$ is not, regardless of the values of $y$ and $b$. This is noteworthy because it is an example where $\lambda_{\rm{q}}$ has additional zeros 
that depend only on the input and output of a single party, even though $\lambda$ has no such zeros.

The above remarks show that there are many homomorphism games which are not q-reflexive. 
However, the proof of the above fact about walks that was given for q-strategies in~\cite{qhoms} works just as well for qc-strategies. So these homomorphism games will also not be qc-reflexive. Next we will see an example of a game that is not even ns-reflexive.
\end{example}

\begin{example}\label{ex_nsnonr}
The \emph{$(G,H)$-isomorphism game} is similar to the homomorphism game described in Example~\ref{ex_qcnonr} but, 
in it, both adjacency and non-adjacency must be preserved. More precisely, in the $(G,H)$-isomorphism game, the input and output sets for both players are all equal to $V(G) \cup V(H)$, where the vertex sets of the two graphs are assumed to be disjoint. If Alice (Bob) receives a vertex from $G$, she (he) must respond with a vertex from $H$, and vice versa. If this first condition of the game is met, then Alice either receives as input or sends as output a vertex $g_A$ of $G$. Similarly, Alice sends or receives a vertex $h_A$ of $H$. We can define $g_B$ and $h_B$ analogously for Bob. The remaining rule of the $(G,H)$-isomorphism game is that these four vertices must satisfy $\text{rel}(g_A,g_B) = \text{rel}(h_A,h_B)$, where $\text{rel}$ is a function determining whether two vertices are equal, adjacent, or distinct and non-adjacent. 
All of these conditions are encoded in the rule function $\lambda$ for the $(G,H)$-isomorphism game.

A partition $P_1, \ldots, P_k$ of the vertex set $V(G)$ 
of a graph $G$ is said to be equitable if there exist integers $c_{ij}$ for $i,j \in \{1, \ldots, k\}$ such that each vertex in $P_i$ is adjacent to precisely $c_{ij}$ vertices in $P_j$. We refer to the numbers $c_{ij}$ as the \emph{partition numbers} of the partition $P_1, \ldots, P_k$. For example, the trivial partition into singletons is always equitable for any graph. Also, the single 
element partition is equitable for a graph $G$ if and only if every vertex of $G$ has the same number of neighbors, i.e.
if and only of $G$ is \emph{regular}. In general, given two equitable partitions of a graph $G$, their \emph{join} will be an equitable partition of $G$. It follows that any graph has a unique (up to permutation of the parts) coarsest equitable partition. Suppose that $P_1, \ldots, P_k$ and $Q_1, \ldots, Q_\ell$ are the coarsest equitable partitions of $G$ and $H$ with partition numbers $c_{ij}$ for $i,j \in \{1, \ldots, k\}$, and $d_{ij}$ for $i,j \in \{1, \ldots, \ell\}$, respectively. 
We say that $G$ and $H$ have \emph{common coarsest equitable partitions} if $k = \ell$, and (up to some permutation of the indices) $|P_i| = |Q_i|$ for all $i = 1, \ldots, k$, and $c_{ij} = d_{ij}$ for all $i,j \in \{1, \ldots, k\}$.
In~\cite{qiso}, it was shown that there exists a perfect ns-strategy for the $(G,H)$-isomorphism game if and only if the graphs have \emph{common coarsest equitable partitions}. 

Now suppose that $G$ and $H$ are graphs with common coarsest equitable partitions $P_1, \ldots, P_k$ and $Q_1, \ldots, Q_k$ respectively. Though it is not explicitly stated in~\cite{qiso}, it follows from the proof of 
Lemma~4.2 therein and the proof of \cite[Theorem 2.2]{RSU}, 
that in any perfect ns-strategy $p$ for the $(G,H)$-isomorphism game, the marginal probability $p(h|g) = p(h,h|g,g)$
vanishes unless $g \in P_i$ and $h \in Q_i$ for some $i \in \{1, \ldots, k\}$. Therefore, we obtain that $\lambda_{\rm {ns}}(g,g',h',h') = 0$ if $g \in P_i$ and $h \in Q_j$ with $i \ne j$, regardless of the values of $g'$ and $h'$. Thus, if $k \ge 2$, then the $(G,H)$-isomorphism game is not ns-reflexive. This occurs if and only if the graphs $G$ and $H$ are not regular.
\end{example}


\section{Imitation games\label{Section:games}}

In this section, we introduce a new class of games that we call {\it imitation games}, give examples, and establish some first properties.

\begin{definition}
\label{Definition:imitation}
A game $\cl G = (X,Y,A,B,\lambda)$ will be called
an \emph{imitation game} if
\begin{itemize}
\item[(a)] for every $x\in X$ and $a,a^{\prime }\in
A $ with $a\neq a^{\prime }$, there exists $y\in Y$ such that
\begin{equation*}
\sum_{b\in B}\lambda \left( a,b|x,y\right) \lambda \left( a^{\prime},b|x,y\right) = 0, 
\end{equation*}
and 
\item[(b)] 
for every $y\in Y$ and $b,b^{\prime }\in B$ with $b\neq b^{\prime }$, there exists $x\in X$ such that
\begin{equation*}
\sum_{a\in A}\lambda \left( a,b|x,y\right) \lambda \left( a,b^{\prime}|x,y\right) =0\text{.}
\end{equation*}
\end{itemize}
\end{definition}

Conditions (a) and (b) in Definition \ref{Definition:imitation} assert that, in some sense, in a perfect
strategy for $\mathcal{G}$, Alice's answers are
completely determined by Bob's answers and vice versa.

Given an imitation
game $\mathcal{G}$, we define a corresponding C*-algebra $C^{\ast }\left( 
\mathcal{G}\right) $ as follows.

\begin{definition}
\label{Definition:algebra}
Let $\cl G$ be an imitation game. 
The C*-algebra $C^{\ast }\left( \mathcal{G}\right)$ of the imitation game $\mathcal{G}$ is the universal unital C*-algebra
generated by families $\left( p_{x,a}\right) _{x\in X,a\in A}$ and $\left(
q_{y,b}\right) _{y\in Y,b\in B}$ of projections satisfying the relations:

\begin{itemize}
\item[(a)] $\sum_{a\in A} p_{x,a} = 1$ for every $x\in X$;

\item[(b)] $\sum_{b\in B} q_{y,b} = 1$ for every $y\in Y$;

\item[(c)] $\lambda _{\mathcal{G}}(x,y,a,b) = 0$ implies $p_{x,a} q_{y,b} = 0$,  
for any $(x,y,a,b) \in X\times Y\times A\times B$.
\end{itemize}
\end{definition}

For convenience, in Definition \ref{Definition:algebra} we consider the zero
C*-algebra $\left\{ 0\right\} $ to be a unital C*-algebra. It is possible
that $C^{\ast }\left( \mathcal{G}\right) =\left\{ 0\right\} $, in which case
we say that the C*-algebra of the game is zero. 

We now provide examples of
several classes of games that are particular instances of imitation games.

\begin{example}
\label{Example:synchronicity}Suppose that $A=B$ and $X=Y$. 
The \emph{synchronicity game} $\mathcal{G}^{\mathrm{s}}$ is
obtained by setting $\lambda _{\mathcal{G}^{\mathrm{s}}}\left(
a,b,x,y\right) =0$ if and only if $x=y$ and $a\neq b$. The correlations in 
$\cl C\left( \mathcal{G}^{\mathrm{s}}\right) $ are called \emph{synchronous
correlations }and denoted by $\cl C^{\mathrm{s}}$.
\end{example}

\begin{example}
\label{Example:synchronous}
Any game $\mathcal{G}$, with $X = Y$ and $A = B$, 
harder than the synchronicity game, is called \emph{synchronous}.
It is clear that any synchronous game is an imitation game.
In this case, the C*-algebra of the game $\mathcal{G}$ as in Definition \ref{Definition:algebra} coincides with the 
C*-algebra of the game as defined in 
\cite{helton_algebras_2017}. The class of synchronous games contains in
particular the graph coloring games and the graph homomorphism games
introduced and studied in \cite%
{cameron_quantum_2007,qhoms,Rthesis,paulsen_estimating_2016,paulsen_quantum_2015,ortiz_quantum_2016}%
. In the case of the graph homomorphism game, the game C*-algebra as in
Definition \ref{Definition:algebra} recovers the C*-algebra considered in 
\cite[Section 4]{ortiz_quantum_2016}.
\end{example}

\begin{example}
\label{Example:BCS}Suppose that $Y$ is a finite set of \emph{variables} $%
\left\{ v_{1},\ldots ,v_{n}\right\} $. A \emph{binary constraint} in $%
v_{1},\ldots ,v_{n}$ is an expression of the form $f((v_{j})_{j\in V})=1$
for some $V\subseteq \left\{ 1,2,\ldots ,n\right\}$, where $f:\left\{
-1,1\right\} ^{V}\rightarrow \left\{ -1,1\right\} $ is a function. A \emph{%
binary constraint system} $\mathcal{S}$ is a set $X$ of such binary
constraints with the property that every variable appears in some
constraint. In \cite{cleve_characterization_2014}, a game $\mathcal{G}_{%
\mathcal{S}}$ has been associated with a binary constraint system, as
follows. The input sets are $X$ and $Y$ as above, while the output sets are $%
B=\left\{+1,-1\right\}$ and 
$A = \left\{ \left( a_{i}\right) _{i\in V}:V\subseteq \left\{ 1,2,\ldots ,n\right\} ,a_{i}\in \left\{+1,-1\right\}\right\}$. 
The payoff function is defined by 
$\lambda \left(x,v_{j},\left(a_{i}\right) _{i\in V},b\right) =1$ if and only if $x$ is a
constraint of the form $f((v_{k})_{k\in V}) =1$, the tuple $(a_k)_{k\in V}$ satisfies the constraint $x$, $j\in V$, and $b=a_{j}$. A particular example of a binary constraint system game is the Mermin-Peres magic square game \cite{holevo}
\end{example}

\begin{example}
\label{Example:variables}
Suppose that $C$ is a finite set of possible \emph{values}
for some \emph{variables} $v_{1},\ldots ,v_{n}$. As before, we let $X$ and $Y$ represent
the set of possible questions for Alice and Bob, respectively. To each $z\in
X\cup Y$ we assigns a subset $V_{z}$ of $\left\{ 1,2,\ldots ,n\right\} $,
in such a way that, for any $i\in \left\{ 1,2,\ldots ,n\right\}$, there exist 
$x\in X$ and $y\in Y$ such that $i\in V_{x}\cap V_{y}$. We let the sets $A$ and $B$ be
both equal to the set of tuples 
$\left(c_{i}\right) _{i\in V}$ for some $V\subseteq \left\{ 1,2,\ldots ,n\right\}$,
which we can think of as a valuation of the variables $v_{i}$ for $i\in V$.
A \emph{\ variable assignment game} $\cl G$ on $(X,Y,A,B)$ is any game whose 
payoff function $\lambda _{\mathcal{G}}:X\times Y\times A\times B \rightarrow \left\{0,1\right\} $ satisfies 
the conditions 
\begin{itemize}
\item[(a)] $\lambda(x, y, (a_{i})_{i\in V}, (b_{j}) _{j\in W}) =1$ implies 
$V=V_{x}$ and $W = V_{y}$, and 
\item[(b)] $a_{i} = b_{i}$ for every $i\in V_{x}\cap V_{y}$. 
\end{itemize}
It is clear that any binary constraint system game is a variable assignment game.

We show that every variable assignment game is an imitation game.
Suppose that $x\in X$ and $a,a'\in A$ with $a\neq a'$, say $a = (a_i)_{i\in V}$ and $a' = (a'_i)_{i\in V'}$, 
where $V$ and $V'$ are subsets of $\{1,\dots,n\}$. 
If $V\neq V'$ then either $V\neq V_x$, in which case choosing an arbitrary $y\in Y$ we have that 
$\lambda(x,y,a,b) = 0$ for all $b\in B$, or $V'\neq V_x$, in which case choosing an arbitrary $y\in Y$ we have that 
$\lambda(x,y,a',b) = 0$ for all $b\in B$. 
Suppose that $V = V'$, and choose $i\in V$ such that $a_i \neq a_i'$. Let $y\in Y$ be such that 
$i\in V_x\cap V_y$. Let $b = (b_j)_{j\in W}\in B$. If $W\neq V_y$ then clearly
$\lambda(x,y,a,b) = 0$. If, on the other hand, $W = V_y$ then 
$\lambda(x,y,a,b) \lambda(x,y,a',b) = 0$. By symmetry, $\cl G$ is an imitation game. 
\end{example}

\begin{example}
\label{Example:mirror}
We call a non-signalling game $\cl G = (X,Y,A,B,\lambda)$ a \emph{mirror game} if 
there exist functions $\xi : X\to Y$ and $\eta : Y\to X$ such that
$$\lambda(x,\xi(x),a,b) \lambda(x,\xi(x),a',b) = 0, \ \ \ x\in X, \ a,a'\in A, \ b\in B, \ a\neq a'$$
and 
$$\lambda(\eta(y),y,a,b) \lambda(\eta(y),y,a,b') = 0, \ \ \ y\in Y, \ a\in A, \ b,b'\in B, \ b\neq b'.$$
Clearly, every mirror game is an imitation game, and the difference between the two classes consists in that,
given $x\in X$ (resp.\ $y\in Y$) and $a,a'\in A$ (resp.\ $b,b'\in B$) with $a\neq a'$ (resp.\ $b\neq b'$),
the element $y\in Y$ (resp.\ $x\in X$) satisfying condition (a) (resp.\ (b)) in Definition \ref{Definition:imitation}
depends on $a$ and $a'$ (resp.\ $b$ and $b'$) for general imitation games, and is independent of 
them for mirror games. 
\end{example}

\begin{example}
\label{Example:unique}
Recall that a non-signalling game $\cl G = (X,Y,A,B,\lambda)$ is called \emph{unique} (see e.g.\ \cite{rao})
if for every $(x,y)\in X\times Y$ there exists a bijection $\phi_{x,y} : A\to B$ such that 
$$\lambda(x,y,a,b) = 1 \ \Longleftrightarrow \ b = \phi_{x,y}(a).$$
Every unique game is a mirror game; indeed, given $x\in X$, 
any element $y$ of $Y$ satisfies (a) in Definition \ref{Definition:imitation} for any choice of $a,a'\in A$ with $a\neq a'$; 
the claim follows by symmetry.
\end{example}

\begin{proposition}
\label{Prop:variable} Let $\mathcal{G} = (X,Y,A,B,\lambda)$ be a variable assignment
game with $n$ variables and set of variable values $C$. 
Its game C*-algebra $C^{\ast }\left( \mathcal{G}\right)$ is *-isomorphic to 
the universal C*-algebra generated by a family $\{e_{i,c} : i = 1,\dots,n, c\in C\}$ of projections 
subject to the relations
\begin{itemize}
\item[(a)] $\sum_{c\in C} e_{i,c} = 1$, $i=1,2,\ldots ,n$;

\item[(b)] $e_{i,c} e_{j,d} = e_{j,d} e_{i,c}$ for all $c,d\in C$, 
whenever there exists $z\in X\cup Y$ such that $i,j\in V_{z}$;

\item[(c)] the projections $\prod_{i\in V_{x}}e_{i,a_{i}}$ and $\prod_{j\in V_{y}}e_{j,b_{j}}$ are
orthogonal whenever 
$\lambda(x,y,\left( a_{i}\right) _{i\in V_{x}},\left(b_{i}\right) _{i\in V_{y}})= 0$.
\end{itemize}
\end{proposition}

\begin{proof}
Let $\mathcal{U}$ be the universal C*-algebra
as in the statement, and fix a faithful *-representation $C^{\ast }\left( \mathcal{G}\right) \subseteq
\cl B\left( \mathcal{H}\right)$. 
By the universal property of $C^*(\cl G)$, there exists a canonical *-homomorphism 
$\pi :C^{\ast}\left( \mathcal{G}\right) \rightarrow \mathcal{U}$ such that 
\begin{equation*}
\pi(p_{x,\left( a_{i}\right) _{i\in V}})= \left\{ 
\begin{array}{cc}
\prod_{i\in V_{x}}e_{i,a_{i}} & \text{if }V=V_{x}\text{,} \\ 
0 & \text{otherwise}
\end{array}
\right.
\end{equation*}
and
\begin{equation*}
\pi(q_{y,\left( b_{i}\right) _{i\in V}})= \left\{ 
\begin{array}{cc}
\prod_{i\in V_{y}}e_{i,a_{i}} & \text{if }V=V_{y}\text{,} \\ 
0 & \text{otherwise}
\end{array}
\right. \text{.}
\end{equation*}
We now define a *-homomorphism $\rho :\mathcal{U}\rightarrow
C^{\ast }\left( \mathcal{G}\right) $ which is the inverse of $\pi $. 
If $x\in X$, $i\in V_{x}$, and $c\in C$, let
\begin{equation*}
g_{i,c}^{x}=\sum \left\{ p_{x,a}:a\in C^{V_{x}},a_{i}=c\right\} \text{.}
\end{equation*}
Similarly if $y\in Y$, $i\in V_{y}$, and $c\in C$, let
\begin{equation*}
h_{i,c}^{y} = \sum \left\{ q_{y,b} : b\in C^{V_{y}},b_{i}=c\right\} \text{.}
\end{equation*}
It is clear that $g_{i,c}^x$ and $h_{i,c}^y$ are projections. 
Fix $x\in X$ and $y\in Y$. Observe that, for every $i\in V_{x}$ and $j\in V_{y}$,
\begin{equation}\label{eq_oneone}
\sum_{c\in C}g_{i,c}^{x} = \sum_{c\in C} h_{j,c}^{y} = 1\text{.}
\end{equation}
Suppose that $i\in V_{x}\cap V_{y}$. Since $p_{x,a}q_{y,b}=0$
whenever $a_{i}\neq b_{i}$, we have
\begin{equation}\label{eq_newone}
1 = \left(\sum_{c\in C}g_{i,c}^{x}\right) \left(\sum_{d\in C} h_{j,d}^{y}\right) = \sum_{c\in C}g_{i,c}^{x} h_{i,c}^{y}\text{.}
\end{equation}
Let $\left\vert \xi \right\rangle \in \mathcal{H}$ be a unit vector,
$\left\vert \alpha \right\rangle =\left(g_{i,c}^{x}\left\vert \xi \right\rangle \right) _{c\in C}$
and $\left\vert \beta \right\rangle =\left(h_{i,c}^{y}\left\vert \xi \right\rangle \right) _{c\in C}$. 
By (\ref{eq_oneone}), $\left\vert \alpha \right\rangle$ and $\left\vert \beta \right\rangle$ are unit vectors in 
$\oplus _{c\in C}\mathcal{H}$ while, by (\ref{eq_newone}), they satisfy the relation 
$\left\langle \alpha |\beta \right\rangle =1$. Thus $\alpha =\beta$, 
that is, $g_{i,c}^{x}\left\vert \xi \right\rangle = h_{i,c}^{y}\left\vert \xi
\right\rangle $ for every $\left\vert \xi \right\rangle \in \mathcal{H}$ and
hence $g_{i,c}^{x} = h_{i,c}^{y}$. Therefore $g_{i,c}^{x} = h_{i,c}^{y}$ for every 
$i\in V_{x}\cap V_{y}$ and every $c\in C$. We conclude that $g_{i,c}^{x} = g_{i,c}^{x^{\prime }} := g_{i,c}$ 
for any $x,x^{\prime }\in X$, $i\in V_{x}\cap V_{x^{\prime}}$, and $c\in C$. 
Since $(p_{x,a})_{a\in A}$ and $(q_{y,b})_{b\in B}$ are PVM's for all $x\in X$ and $y\in Y$, 
the family $\{g_{i,c} : i = 1,\dots,n, c\in C\}$ satisfies condition (b). 
Condition (c) is straightforward. 
Therefore the map $e_{i,c}\mapsto g_{i,c}$ extends to a
*-homomorphism $\rho :\mathcal{U}\rightarrow C^{\ast }\left( \mathcal{G}\right)$,
which is easily seen to be the inverse of $\pi$.
\end{proof}

A particular instance of a\emph{\ }binary constraint system is a \emph{%
linear }binary constraint system as considered in \cite{cleve_perfect_2016}%
. This is a binary constraint system $\mathcal{S}$ where the functions $%
f:\left\{ -1,1\right\} ^{V}\rightarrow \left\{ -1,1\right\} $ are of the
form $f\left( \left( \lambda _{i}\right) _{i\in V}\right) =\left( -1\right)
^{\rho }\prod_{i\in V}\lambda _{i}$ for some $\rho \in \left\{ 0,1\right\} $%
. The \emph{solution group }$\Gamma \left( \mathcal{S}\right) $ associated
to such a linear binary constraint system as in \cite{cleve_perfect_2016,
s_2016} is defined to be the group generated by
involutions $u_{1},\ldots ,u_{n},J$ subject to the following relations: $J$
commutes with $u_{1},\ldots ,u_{n}$, and $u_{i},u_{j}$ commute whenever a 
constraint of the form $\left( -1\right) ^{\rho }\prod_{k\in V}\lambda _{k}=1$, with $i,j\in V$, 
is an element of the input set $X$,
in which case $J^{\rho }\prod_{k\in V}u_{k}=1$.

We specialise Example \ref{Example:variables} to the case of a binary constraint system $\cl S$ and describe 
the canonical non-signalling game $\cl G_{\cl S}$ associated with it.
We have that $X$ is the set of all constraints and $C=\left\{ +1,-1\right\} $. For $x\in X$, $V_{x}$ is the
set of indices of variables that appear in the constraint $x$. We set 
$Y = \left\{ 1,2,\ldots ,n\right\} $ and $V_{y}=\left\{ y\right\} $ for $y\in Y$. 
The payoff function $\lambda$ is obtained by letting, for $x\in X$ and $j\in V_{x}$, where 
$x$ is the constraint $\left( -1\right) ^{\rho _{x}}\prod_{i\in V_{x}}\lambda
_{i}=1$, $\lambda \left( \left( a_{i}\right) _{i\in V_{x}},a_{j},x,j\right)
=1$ if and only if $\left( -1\right) ^{\rho _{x}}\prod_{i\in V_{x}}a_{x}=1$.

We denote by $C^{\ast }(\Gamma(\mathcal{S})) $ the
full group C*-algebra of the solution group $\Gamma(\mathcal{S})$ 
and identify $\Gamma(\mathcal{S})$ with a subgroup
of the unitary group of $C^*(\Gamma(\mathcal{S}))$. 
A natural question that arises, and is addressed in the next proposition, is 
what the relation between $C^*(\cl G_{\cl S})$ and $C^*(\Gamma(\mathcal{S}))$ is.

\begin{proposition}
\label{Prop:solution}
Let $\mathcal{S}$ be a linear binary constraint system with 
a corresponding non-signalling game $\mathcal{G}_{\mathcal{S}}$ and solution
group $\Gamma \left( \mathcal{S}\right) $. 
The game C*-algebra $C^{\ast}(\cl G_{\cl S}) $ is *-isomorphic to the quotient of the full
group C*-algebra $C^{\ast }\left( \Gamma \left( \mathcal{S}\right) \right)$ by the relation $J+1=0$.
\end{proposition}

\begin{proof}
Let $J$ be the closed two sided ideal of $C^*(\Gamma(\mathcal{S}))$ generated by $J+1$ and 
$\mathcal{A} = C^*(\Gamma(\mathcal{S}))/J$. 
We denote by $\hat{u}%
_{1},\ldots ,\hat{u}_{n}$ the images inside $\mathcal{A}$ of the canonical
generators of $\Gamma \left( \mathcal{S}\right) $ under the canonical
quotient mapping $C^{\ast }\left( \Gamma \left( \mathcal{S}\right) \right)
\rightarrow \mathcal{A}$. We consider the description of $C^{\ast }\left( 
\mathcal{G}\right) $ as in Proposition \ref{Prop:variable}. 
(Recall that a 
binary constraint system game is, in particular, a variable assignment game
in the sense of Example \ref{Example:variables}, and therefore Proposition 
\ref{Prop:variable} applies.) Consider the assignment $u_{i}\mapsto
e_{i,+1}-e_{i,-1}$ and $J\mapsto -1$. It is easy to verify that this
defines a unitary representation of $\Gamma \left( \mathcal{S}\right) $, and
hence it extends to a *-homomorphism $\pi : C^{\ast }\left( \Gamma \left( \mathcal{%
S}\right) \right) \rightarrow C^{\ast }\left( \mathcal{G}\right) $. Since
such a *-homomorphism maps $J$ to $-1$, it induces a *-homomorphism $%
\mathcal{A}\rightarrow C^{\ast }\left( \mathcal{G}\right) $. Conversely, the
assignment $e_{i,+1}\mapsto \hat{u}_{i}^{+}$ and $e_{i,-1}\mapsto \hat{u}%
_{i}^{-}$ defines a *-homomorphism $C^{\ast }\left( \mathcal{G}\right)
\rightarrow \mathcal{A}$, where $\hat{u}_{i}^{+}$ and $\hat{u}_{i}^{-}$
denote the spectral subspaces of the selfadjoint unitary $\hat{u}_{i}$
associated with $+1$ and $-1$, respectively. It is clear from the definition
that the *-homomorphisms $\pi$ and $\rho$ are inverses of each other.
\end{proof}


\section{Perfect strategies for imitation games}\label{s_wsig}

In \cite{helton_algebras_2017} the existence of various types of perfect strategies for synchronous games was given characterisations in terms of types of traces on the game C*-algebra.
In this section, we extend those results to imitation games. Given an imitation game $\mathcal{G}$, we characterise the
elements of the set $\cl C_{\mathrm{qc}}\left( \mathcal{G}\right) $ of perfect quantum
commuting strategies, and of the set $\cl C_{\mathrm{q}}\left( \mathcal{G}\right)$ of 
perfect quantum strategies, 
for the game $\mathcal{G}$, in terms of traces on the game C*-algebra $C^{\ast }\left(\mathcal{G}\right)$.

\begin{theorem}
\label{Thm:qc}
Let $\mathcal{G} = (X,Y,A,B,\lambda)$ be an imitation game and $p$ be a non-signalling 
correlation on $(X,Y,A,B)$. 
The following statements are equivalent:

\begin{itemize}
\item[(a)] $p\in C_{\mathrm{qc}}\left( \mathcal{G}\right) $;

\item[(b)] the C*-algebra of the game $C^{\ast }\left( \mathcal{G}\right) $ is
nonzero, and there exists a tracial state $\tau $ on $C^{\ast }\left( 
\mathcal{G}\right) $ such that $p\left( a,b|x,y\right) =\tau \left(
p_{x,a}q_{y,b}\right) $.
\end{itemize}
\end{theorem}

\begin{proof}
(i)$\Rightarrow$(ii) Fix $p\in C_{\mathrm{qc}}(\lambda)$. 
By Corollary \ref{th_qcgen}, there exists a separable Hilbert spaces $%
\mathcal{H}$, a unit vector $\left\vert \xi \right\rangle \in \mathcal{H}%
\otimes \mathcal{H}$, and PVMs $\left( P_{x,a}\right) _{a\in A}$ and $\left(
Q_{y,b}\right) _{b\in B}$ on $\mathcal{H}$ for $x\in X$ and $y\in Y$ such
that 
$$p\left( a,b|x,y\right) =\left\langle \xi |P_{x,a}Q_{y,b}|\xi\right\rangle, \ \ \ (x,y,a,b) \in X\times Y\times A\times B.$$ 
Let $\mathcal{A}$ (resp.\ $\cl B$) be the C*-algebra generated by $\left\{ P_{x,a}:x\in
X,a\in A\right\} $ (resp.\ $\left\{Q_{y,b} : y\in Y, b\in B\right\}$). 
Let $\mathcal{M}$ (resp.\ $\cl N$) be the WOT-closure of $\mathcal{A}$
(resp.\ $\mathcal{B}$). Observe that $\mathcal{M}\subseteq \mathcal{N}^{\prime }$.
For $x\in X$, $y\in Y$ and $b\in B$, set 
\begin{equation*}
\Pi _{y,b}^{x} = \sum_{a\in A,\lambda \left(x,y,a,b\right) =1}P_{x,a}\text{.}
\end{equation*}
Clearly, $\Pi _{y,b}^{x}$ is a projection in $\cl A$. 
We have
\begin{eqnarray*}
\left\langle \xi | Q_{y,b}|\xi \right\rangle 
& = & 
\sum_{a\in A} \left\langle \xi | P_{x,a} Q_{y,b}|\xi \right\rangle = 
\sum_{a\in A, \lambda(x,y,a,b) = 1} \left\langle \xi | P_{x,a} Q_{y,b}|\xi \right\rangle\\
& = & 
\left\langle \xi | \Pi_{y,b}^x Q_{y,b}|\xi \right\rangle.
\end{eqnarray*}
Since $(I - \Pi_{y,b}^x) Q_{y,b}$ is an idempotent, this shows that 
$$\|(I - \Pi_{y,b}^x) Q_{y,b}|\xi \rangle\|^2 = \left\langle \xi | (I - \Pi_{y,b}^x) Q_{y,b}|\xi \right\rangle = 0,$$
that is,
\begin{equation}\label{eq_piq}
\Pi _{y,b}^{x}Q_{y,b}\left\vert \xi \right\rangle = Q_{y,b}\left\vert \xi \right\rangle, \ \ \  x\in X, y\in Y, b\in B. 
\end{equation}
By assumption, for $b\neq b^{\prime }$ and $y\in Y$ there exists 
$x\in X$ such that $\sum_{a\in A}\lambda \left(x,y,a,b\right) \lambda \left(x,y,a,b'\right) =0$. For such a choice of $x,y,b,b^{\prime }$ we
have 
\begin{equation*}
\left\langle \xi |\Pi _{y,b}^{x}Q_{y,b^{\prime }}|\xi \right\rangle
=\sum_{a\in A,\lambda \left(x,y,a,b\right) =1}\left\langle \xi
|P_{x,a}Q_{y,b^{\prime }}|\xi \right\rangle = 0;
\end{equation*}
henceforth $\Pi _{y,b}^{x}Q_{y,b^{\prime }}\left\vert \xi \right\rangle =0$.

Let $\Pi _{y,b}$ be the projection onto the intersection of the
ranges of $\Pi _{y,b}^{x}$ for $x\in X$. 
Since the projections in $\mathcal{M}$ form a complete sublattice of the
lattice of projections in $\cl B\left( \mathcal{H}\right)$, we have that $\Pi _{y,b}\in \mathcal{M}$. 
By (\ref{eq_piq}),  
\begin{equation}\label{eq_ybqyb}
\Pi _{y,b}Q_{y,b}\left\vert\xi \right\rangle = Q_{y,b}\left\vert \xi \right\rangle. 
\end{equation}
Suppose now that 
$b^{\prime }\in B$ and $b'\neq b$. By the preceding paragraph, there
exists $x\in X$ such that $\Pi _{y,b}^{x}Q_{y,b^{\prime }}\left\vert \xi
\right\rangle = 0$. 
Thus $\Pi _{y,b}Q_{y,b^{\prime}}\left\vert \xi \right\rangle = 0$ whenever $b^{\prime }\neq b$. Thus, using (\ref{eq_ybqyb}), we obtain
\begin{equation}\label{eq_piqsame}
\Pi _{y,b}\left\vert \xi \right\rangle = \Pi _{y,b}\left(\sum_{b^{\prime }\in
B}Q_{y,b^{\prime }}\right)\left\vert \xi \right\rangle =\Pi _{y,b}Q_{y,b}\left\vert
\xi \right\rangle = Q_{y,b}\left\vert \xi \right\rangle \text{.}
\end{equation}

Similarly, define 
\begin{equation*}
\Xi _{x,a}^{y}=\sum_{b\in B,\lambda \left( a,b,x,y\right) =1}Q_{y,b},
\end{equation*}%
let $\Xi _{x,a}\in \mathcal{N}$ be the projection onto the intersection of
the ranges of $\Xi _{x,a}^{y}$ for $y\in Y$, and show that
\begin{equation}\label{eq_xip}
\Xi _{x,a}\left\vert \xi \right\rangle =P_{x,a}\left\vert \xi \right\rangle, \ \ \ x\in X,  a\in A.
\end{equation}

Let $\mathcal{K}_{\mathcal{M}}$ (resp.\ $\mathcal{K}_{\mathcal{N}}$) 
be the closure of the span of 
$\left\{ A\left\vert \xi \right\rangle : A\in \mathcal{M}\right\} $ 
(resp.\ $\left\{B\left\vert \xi \right\rangle :B\in \mathcal{N}\right\}$). 
Suppose that $x_{1},\ldots ,x_{n}\in X$ and $a_{1},\ldots ,a_{n}\in
A$. Then 
\begin{equation*}
P_{x_{1},a_{1}}\cdots P_{x_{n},a_{n}}\left\vert \xi \right\rangle =\Xi
_{x_{n},a_{n}}\cdots \Xi _{x_{1},a_{1}}\left\vert \xi \right\rangle \in 
\mathcal{K}_{\mathcal{N}};
\end{equation*}%
therefore $\mathcal{K}_{\mathcal{M}}\subseteq \mathcal{K}_{\mathcal{N}}$.
Similarly, $\mathcal{K}_{\mathcal{N}}\subseteq \mathcal{K}_{%
\mathcal{M}}$. Set $\mathcal{K}:=\mathcal{K}_{\mathcal{M}}=\mathcal{K}_{%
\mathcal{N}}$. 
Clearly, $\cl K$ is invariant under the projections $P_{x,a}$ and $Q_{y,b}$; 
after replacing $P_{x,a}$ and $Q_{y,b}$ with their restrictions
to $\mathcal{K}$, we can thus assume that $\mathcal{K}=\mathcal{H}$.


Fix $y\in Y$ and $b,b^{\prime }\in B$ with $b\neq b^{\prime }$. By assumption,
there exists $x\in X$ such that $\sum_{a\in A}\lambda \left(x,y,a,b\right)
\lambda \left(x,y,a,b'\right) = 0$. 
We thus have that $\Pi_{y,b}^x \Pi_{y,b'}^x = 0$. Since $\Pi_{y,b}\leq \Pi_{y,b}^x$ and 
$\Pi_{y,b'}\leq \Pi_{y,b'}^x$, we conclude that 
$\Pi _{y,b} \Pi _{y,b^{\prime }} = 0$.
On the other hand, if $Z\in \mathcal{N}$ then, using (\ref{eq_piqsame}), we have 
\begin{equation*}
\sum_{b\in B}\Pi _{y,b}Z\left\vert \xi \right\rangle =
\sum_{b\in B} Z \Pi _{y,b}\left\vert \xi \right\rangle = 
Z\sum_{b\in
B}Q_{y,b}\left\vert \xi \right\rangle =Z\left\vert \xi \right\rangle \text{.}
\end{equation*}%
Since this is true for every $Z\in \mathcal{N}$ we conclude that $\sum_{b\in B}\Pi _{y,b} = I$.

We show that the vector state $\tau : \cl A\to \bb{C}$, given by  
$\tau(A) = \left\langle \xi |A|\xi\right\rangle$, is tracial. 
Indeed, for 
$x_{1},\ldots ,x_{n},x_{1}^{\prime },\ldots ,x_{m}^{\prime }\in X$ and 
$a_{1},\ldots ,a_{n}$, $a_{1}^{\prime },\ldots ,a_{m}^{\prime }\in A$, using (\ref{eq_xip}), we have 
\begin{eqnarray*}
&&\left\langle \xi |\left( P_{x_{1},a_{1}}\cdots P_{x_{n},a_{n}}\right)
(P_{x_{1}^{\prime },a_{1}^{\prime }}\cdots P_{x_{m}^{\prime },a_{m}^{\prime
}})|\xi \right\rangle \\
&=&\left\langle \xi |\Xi _{x_{m}^{\prime },a_{m}^{\prime }}\left(
P_{x_{1},a_{1}}\cdots P_{x_{n},a_{n}}\right) (P_{x_{1}^{\prime },a_{1}^{\prime
}}\cdots P_{x_{m-1}^{\prime },a_{m-1}^{\prime }})|\xi \right\rangle \\
&=&\left\langle \Xi _{x_{m}^{\prime },a_{m}^{\prime }}\xi |\left(
P_{x_{1},a_{1}}\cdots P_{x_{n},a_{n}}\right) (P_{x_{1}^{\prime },a_{1}^{\prime
}}\cdots P_{x_{m-1}^{\prime },a_{m-1}^{\prime }})|\xi \right\rangle \\
&=&\left\langle P _{x_{m}^{\prime },a_{m}^{\prime }}\xi |\left(
P_{x_{1},a_{1}}\cdots P_{x_{n},a_{n}}\right) (P_{x_{1}^{\prime },a_{1}^{\prime
}}\cdots P_{x_{m-1}^{\prime },a_{m-1}^{\prime }})|\xi \right\rangle \\
&=&\left\langle P _{x_{m}^{\prime },a_{m}^{\prime }}\xi |
\Xi_{x_{m-1}^{\prime },a_{m-1}^{\prime }}\left(
P_{x_{1},a_{1}}\cdots P_{x_{n},a_{n}}\right) (P_{x_{1}^{\prime },a_{1}^{\prime
}}\cdots P_{x_{m-2}^{\prime },a_{m-2}^{\prime }})|\xi \right\rangle \\
&=&\left\langle \xi |\Xi _{x_{m-1}^{\prime },a_{m-1}^{\prime
}}P_{x_{m}^{\prime },a_{m}^{\prime }}\left( P_{x_{1},a_{1}}\cdots
P_{x_{n},a_{n}}\right) (P_{x_{1}^{\prime },a_{1}^{\prime }}\cdots
P_{x_{m-2}^{\prime },a_{m-2}^{\prime }}|\xi \right\rangle \\
&=&\cdots =\left\langle \xi |(P_{x_{1}^{\prime },a_{1}^{\prime }}\cdots
P_{x_{m}^{\prime },a_{m}^{\prime }})\left( P_{x_{1},a_{1}}\cdots
P_{x_{n},a_{n}}\right) |\xi \right\rangle \text{.}
\end{eqnarray*}%
This shows that $\tau $ is a tracial state on $\mathcal{A}$. A similar
argument shows that $\tau $ is a tracial state on $\mathcal{B}$.

We next show that $\tau $ is faithful on $\mathcal{A}$. Suppose that $C\in 
\mathcal{A}$ is positive and assume that $\tau \left( C\right) =0$. 
Then, for every $T\in \mathcal{A}$, we have that
\begin{eqnarray*}
\left\vert \left\langle \xi |T^{\ast }C^{2}T|\xi \right\rangle \right\vert
&=&\tau \left( T^{\ast }C^{2}T\right) =\tau \left( CTT^*C\right) \\
&\leq &\left\Vert T\right\Vert ^{2}\tau (C^{\frac{1}{2}}CC^{\frac{1}{2}%
})\leq \left\Vert T\right\Vert ^{2}\left\Vert C\right\Vert \tau \left(
C\right) =0\text{.}
\end{eqnarray*}%
Therefore $CT\left\vert \xi \right\rangle =0$ for every $T\in \mathcal{A}$, and
hence $C=0$. A similar argument shows that $\tau $ is faithful on $\mathcal{B%
}$.

Suppose now that $\lambda \left(x,y,a,b\right) =0$. Then
\begin{equation*}
0=\left\langle \xi |P_{x,a}Q_{y,b}|\xi \right\rangle =\left\langle \xi
|P_{x,a}\Pi _{y,b}|\xi \right\rangle =\tau \left( P_{x,a}\Pi _{y,b}\right);
\end{equation*}%
therefore, $P_{x,a}\Pi _{y,b} = 0$.

This shows that the C*-algebra of the game $C^{\ast }\left( \mathcal{G}%
\right) $ is nonzero, and the assignment $p_{x,a}\mapsto P_{x,a}$, $%
q_{y,b}\mapsto \Pi_{y,b}$ defines a unital *-homomorphism $\pi :C^{\ast }\left( 
\mathcal{G}\right) \rightarrow \mathcal{A}$. Since $\tau $ is a
tracial state on $\mathcal{A}$, we conclude that $\tau \circ \pi $ is a
tracial state on $C^{\ast }\left( \mathcal{G}\right) $. For $\left(x,y,a,b\right) \in X\times Y\times A\times B$, we have
\begin{equation*}
\left( \tau \circ \pi \right) \left( p_{x,a}q_{y,b}\right) =\tau \left(
P_{x,a}\Pi _{y,b}\right) =\left\langle \xi |P_{x,a}Q_{y,b}|\xi \right\rangle
=p\left(a,b|x,y\right) \text{.}
\end{equation*}%
This concludes the proof of the implication.

(ii)$\Rightarrow$(i) Suppose that there exists a tracial state $\tau $ on $%
C^{\ast }\left( \mathcal{G}\right) $ such that 
$$\tau \left(p_{x,a}q_{y,b}\right) = p\left(a,b|x,y\right), \ \ \ \left(x,y,a,b\right) \in X\times Y\times A\times B.$$ 
Consider the left regular representation $\pi
_{\tau }$ of $C^{\ast }\left( \mathcal{G}\right)$ 
associated with $\tau$,
and the right regular representation $\pi _{\tau }^{\mathrm{op}}$ of $%
C^{\ast }\left( \mathcal{G}\right) ^{\mathrm{op}}$ associated with $\tau$;
recall that they act on the Hilbert space 
$\mathcal{H}=L^{2}\left( C^{\ast }\left(\mathcal{G}\right) ,\tau \right)$ produced via $\tau$ through the 
GNS construction and 
are defined by setting $\pi _{\tau }\left( z\right) \left\vert
w\right\rangle =\left\vert zw\right\rangle $ and $\pi _{\tau }^{\mathrm{op}%
}\left( z\right) \left\vert w\right\rangle =\left\vert wz\right\rangle $ for 
$z,w\in C^{\ast }\left( \mathcal{G}\right) $ (see \cite{brown_c*-algebras_2008}). 
Set $P_{x,a}=\pi \left( p_{x,a}\right)$ and 
$Q_{y,b}=\rho \left( q_{y,b}\right) $, $x\in X$, $y\in Y$, $a\in A$, $b\in B$.
Observe that 
$\left( P_{x,a}\right) _{a\in A}$ and $\left( Q_{y,b}\right) _{b\in B}$ are
PVMs and
\begin{equation*}
\left\langle 1|P_{x,a}Q_{y,b}|1\right\rangle _{L^{2}\left( C^{\ast }\left( 
\mathcal{G}\right) ,\tau \right) }=\left\langle p_{x,a}|q_{y,b}\right\rangle
=\tau \left( p_{x,a}q_{y,b}\right)
\end{equation*}%
for any $\left(x,y,a,b\right) \in X\times Y\times A\times B$.
\end{proof}

\begin{corollary}
\label{Coro:qc}Suppose that $\mathcal{G}$ is an imitation game. The
following assertions are equivalent:

\begin{enumerate}
\item $\mathcal{G}$ has a perfect quantum commuting strategy;

\item the C*-algebra of the game $C^{\ast }\left( \mathcal{G}\right) $ is
nonzero, and it has a tracial state.
\end{enumerate}
\end{corollary}

In view of Proposition \ref{Prop:solution}, Corollary \ref{Coro:qc} recovers 
\cite[Theorem 4]{cleve_perfect_2016} as a particular case.

\medskip

We now turn our attention to other quantum strategies. 
Given an imitation game $\mathcal{G}$, in the next theorem, we characterise the
elements of the set $C_{\mathrm{q}}\left( \mathcal{G}\right) $ in terms of
traces on the game C*-algebra in a fashion analogous to Theorem \ref{Thm:qc}.

\begin{theorem}
\label{Thm:q}
Let $\mathcal{G} = (X,Y,A,B,\lambda)$ be an imitation game and $p$ be a non-signalling 
correlation on $(X,Y,A,B)$. 
The following statements are equivalent:

\begin{itemize}
\item[(i)] $p\in C_{\mathrm{qs}}\left( \mathcal{G}\right) $;

\item[(ii)] $p\in C_{\mathrm{q}}\left( \mathcal{G}\right) $;

\item[(iii)] $p\in C_{\mathrm{qm}}\left( \mathcal{G}\right) $;

\item[(iv)] the game C*-algebra $C^{\ast }\left( \mathcal{G}\right) $ is nonzero,
and there exists a finite-dimensional C*-algebra $\cl F$ with a tracial state $%
\tau $ and a unital *-homomorphism $\pi :C^{\ast }\left( \mathcal{G}\right)
\rightarrow \cl F$ such that 
$$p\left( a,b|x,y\right) =\left( \tau \circ \pi\right) \left( p_{xa}q_{yb}\right), \ \ \ (x,y,a,b)\in X\times Y\times A\times B.$$
\end{itemize}
\end{theorem}

\begin{proof}
(i)$\Rightarrow $(iv) Suppose that $p\in \cl C_{\mathrm{qs}}\left( \mathcal{G}%
\right) $. There exist separable Hilbert spaces $%
\mathcal{H}_{A}$ and $\mathcal{H}_{B}$, a unit vector $\left\vert \xi
\right\rangle \in \mathcal{H}_{A}\otimes \mathcal{H}_{B}$, and PVMs $\left(
P_{x,a}\right) _{a\in A}$ (resp.\ $\left( Q_{y,b}\right) _{b\in B}$) on $\mathcal{H}_A$ 
(resp.\ $\cl H_B$) for $x\in X$ (resp.\ $y\in Y$), such that 
$$p\left( a,b|x,y\right) = \left\langle \xi |P_{x,a}\otimes Q_{y,b}|\xi\right\rangle, \ \ \ \left(x,y,a,b\right) \in X\times Y\times A\times B.$$ 
Following the proof of Theorem \ref{Thm:qc}, for $y\in Y, b\in B$ and $x\in X$, set
\begin{equation*}
\Pi _{y,b}^{x} = \sum_{a\in A,\lambda \left(x,y,a,b\right) =1} P_{x,a}\text{.}
\end{equation*}
Similarly,
for $x\in X,a\in A$ and $y\in Y$, set
\begin{equation*}
\Xi _{x,a}^{y} = \sum_{b\in B,\lambda \left(x,y,a,b\right) =1} Q_{y,b}.
\end{equation*}%
Clearly, $(\Pi _{y,b}^{x})_{b\in B}$ (resp.\ $(\Xi _{x,a}^{y})_{a\in A}$) 
is a family of projections in $\cl B\left( \mathcal{H}_{A}\right)$ 
(resp.\ $\cl B\left( \mathcal{H}_{B}\right)$). 
Let $\Pi_{y,b}$ (resp.\ $\Xi _{x,a}$) be the projection onto the intersection of the ranges of $\Pi
_{y,b}^{x}$ (resp.\ $\Xi _{x,a}^{y}$) for $x\in X$ (resp.\ $y\in Y$).

It follows from the proof of the implication (i)$\Rightarrow $(ii) in Theorem \ref{Thm:qc} 
that $(\Pi _{y,b})_{b\in B}$ (resp.\ $(\Xi _{x,a})_{a\in A}$) is a PVM on $\cl H_A$ (resp.\ $\cl H_B$),
\begin{equation}\label{eq_piten1}
\left( \Pi _{y,b}\otimes I\right) \left\vert \xi \right\rangle =\left(
I\otimes Q_{y,b}\right) \left\vert \xi \right\rangle
\end{equation}
and
\begin{equation}\label{eq_piten2}
\left( I\otimes \Xi _{x,a}\right) \left\vert \xi \right\rangle =\left(
P_{x,a}\otimes I\right) \left\vert \xi \right\rangle \text{.}
\end{equation}
We now follow the arguments in the proof of \cite[Theorem 1]%
{cleve_characterization_2014}. Consider the Schmidt decomposition%
\begin{equation*}
\left\vert \xi \right\rangle =\sum_{i\in I}\alpha _{i}\left( \left\vert \phi
_{i}\right\rangle \otimes \left\vert \psi _{i}\right\rangle \right)
\end{equation*}%
for the unit vector $\left\vert \xi \right\rangle $. By (\ref{eq_piten1}) and (\ref{eq_piten2}),
\begin{equation}\label{eq_phipsi1}
\sum_{i\in I}\alpha _{i}\left(\Pi _{y,b}\left\vert \phi
_{i}\right\rangle \otimes \left\vert \psi _{i}\right\rangle \right)
=\sum_{i\in I}\alpha _{i}\left( \left\vert \phi _{i}\right\rangle \otimes
Q_{y,b}\left\vert \psi _{i}\right\rangle \right)
\end{equation}
and
\begin{equation}\label{eq_phipsi2}
\sum_{i\in I}\alpha _{i}\left( \left\vert \phi _{i}\right\rangle \otimes \Xi
_{x,a}\left\vert \psi _{i}\right\rangle \right) =\sum_{i\in I}\alpha
_{i}\left( P_{x,a}\left\vert \phi _{i}\right\rangle \otimes \left\vert \psi
_{i}\right\rangle \right) \text{.}
\end{equation}
Let $(\mu_b)_{b\in B}$ be unimodular scalars such that the operators
$$U_{y} = \sum_{b\in B} \mu_b \Pi _{y,b} \ \mbox{ and } \ S_y := \sum_{b\in B} \mu_b Q _{y,b}$$
are unitary. Similarly, let 
$(\nu_a)_{a\in A}$ be unimodular scalars such that the operators
$$V_{x} := \sum_{a\in A} \nu_a \Xi _{x,a} \ \mbox{ and } \ T_x := \sum_{a\in A} \nu_a P _{x,a}$$
are unitary.
Equations (\ref{eq_phipsi1}) and (\ref{eq_phipsi2}) imply 
\begin{equation}\label{eq_phipsi3}
\sum_{i\in I}\alpha _{i}\left(U _{y}\left\vert \phi
_{i}\right\rangle \otimes \left\vert \psi _{i}\right\rangle \right)
=\sum_{i\in I}\alpha _{i}\left( \left\vert \phi _{i}\right\rangle \otimes
S_{y}\left\vert \psi _{i}\right\rangle \right)
\end{equation}
and
\begin{equation}\label{eq_phipsi4}
\sum_{i\in I}\alpha _{i}\left( \left\vert \phi _{i}\right\rangle \otimes V_{x}\left\vert \psi _{i}\right\rangle \right) 
=\sum_{i\in I}\alpha_{i}\left(T_{x}\left\vert \phi _{i}\right\rangle \otimes \left\vert \psi
_{i}\right\rangle \right) \text{.}
\end{equation}

Fix a Schmidt coefficient $\alpha $ for $\left\vert \xi \right\rangle $ and
set $I^{\alpha }=\left\{ i\in I:\alpha _{i}=\alpha \right\} $.
Define $\mathcal{K}_{A}^{\alpha }:=\mathrm{span}\left\{ \left\vert \phi
_{i}\right\rangle :i\in I^{\alpha }\right\} $ and $\mathcal{K}_{B}^{\alpha
}:=\mathrm{span}\left\{ \left\vert \psi _{i}\right\rangle :i\in I^{\alpha
}\right\} $. 
By the uniqueness of the Schmidt decomposition, 
(\ref{eq_phipsi1}) and (\ref{eq_phipsi2}), 
we deduce that
\begin{equation*}
\mathrm{span}\left\{U _{y}\left\vert \phi _{i}\right\rangle :i\in
I^{\alpha }\right\} =\mathrm{span}\left\{T_{x}\left\vert \phi
_{i}\right\rangle :i\in I^{\alpha }\right\} = \mathcal{K}_{A}^{\alpha }
\end{equation*}%
and
\begin{equation*}
\mathrm{span}\left\{U_{x}\left\vert \psi _{i}\right\rangle :i\in
I^{\alpha }\right\} =\mathrm{span}\left\{S_{y}\left\vert \psi
_{i}\right\rangle :i\in I^{\alpha }\right\} =\mathcal{K}_{B}^{\alpha }\text{.%
}
\end{equation*}%
Therefore, $\mathcal{K}_{A}^{\alpha }$ is $\Pi _{y,b}$-invariant and $%
P_{x,a}$-invariant for every $x\in X,y\in Y,b\in B,a\in A$. Similarly, $%
\mathcal{K}_{B}^{\alpha }$ is $\Xi _{x,a}$-invariant and $Q_{y,b}$%
-invariant for every $x\in X,y\in Y,b\in B,a\in A$. 

Let $P_{x,a}^{\alpha }$
and $\Pi _{y,b}^{\alpha }$ be the restriction of $P_{x,a}$ and $\Pi
_{y,b}$, respectively, to $\mathcal{K}_{A}^{\alpha }$, which we views 
as operators on $\mathcal{K}_{A}^{\alpha }$. Define  
$Q_{y,b}^{\alpha }, \Xi _{x,a}^{\alpha }\in \cl B\left( \mathcal{K}_{B}^{\alpha
}\right) $ similarly. 
Let also 
$$\left\vert \xi ^{\alpha }\right\rangle = \frac{1}{\sqrt{\left\vert I^{\alpha }\right\vert 
}}\sum_{i\in I^{\alpha }}\left\vert \phi _{i}\right\rangle \otimes
\left\vert \psi _{i}\right\rangle$$ 
be the maximally entangled vector in $\mathcal{K}_{A}^{\alpha }\otimes 
\mathcal{K}_{B}^{\alpha }$. We have that
\begin{equation*}
(\Pi _{y,b}^{\alpha }\otimes I)\left\vert \xi ^{\alpha }\right\rangle
=(I\otimes Q_{y,b}^{\alpha })\left\vert \xi \right\rangle, \ \ \ y\in Y, b\in B.
\end{equation*}
Clearly, $\sum_{b\in B}\Pi _{y,b}^{\alpha }$ is the identity of $\mathcal{K}%
_{A}^{\alpha }$. If $\lambda
\left(x,y,a,b\right) =0$ then
\begin{equation*}
0=\left( P_{x,a}^{\alpha }\otimes Q_{y,b}^{\alpha }\right) \left\vert \xi
^{\alpha }\right\rangle =\left( P_{x,a}^{\alpha }\Pi _{y,b}^{\alpha }\otimes
I\right) \left\vert \xi ^{\alpha }\right\rangle =\left( P_{x,a}^{\alpha }\Pi
_{y,b}^{\alpha }\otimes I\right) \left\vert \xi ^{\alpha }\right\rangle
\end{equation*}%
and hence $P_{x,a}^{\alpha }\Pi _{y,b}^{\alpha }=0$ by \cite[Lemma 2]%
{cleve_characterization_2014}. This shows that the
assignment $p_{x,a}\mapsto P_{x,a}^{\alpha }$ and $q_{y,b}\mapsto \Pi
_{y,b}^{\alpha }$ defines an unital *-homomorphism $\pi ^{\alpha }:C^{\ast
}\left( \mathcal{G}\right) \rightarrow \cl B\left( \mathcal{K}_{A}^{\alpha
}\right) $. Define $\tau ^{\alpha }$ to be the canonical tracial state of $%
\cl B\left( \mathcal{K}_{A}^{\alpha }\right) $, and observe that $\tau ^{\alpha
}\left( T\right) =\left\langle \xi ^{\alpha }|T\otimes I|\xi ^{\alpha
}\right\rangle $ for $T\in \cl B(\mathcal{K}_{A}^{\alpha})$.

Let $\left( \alpha _{n}\right) $ be an enumeration of the Schmidt
coefficients of $\left\vert \xi \right\rangle $. We have%
\begin{eqnarray*}
p\left(a,b|x,y\right) &=&\left\langle \xi |P_{x,a}\otimes Q_{y,b}|\xi
\right\rangle \\
&=&\sum_{i,j}\alpha _{i}\overline{\alpha _{j}}\left( \left\langle \phi
_{j}\right\vert \otimes \left\langle \psi _{j}\right\vert \right) \left(
P_{x,a}\otimes Q_{y,b}\right) \left( \left\vert \phi _{i}\right\rangle \otimes
\left\vert \psi _{i}\right\rangle \right) \\
&=&\sum_{n}\left\vert \alpha _{n}\right\vert ^{2}\left\langle \xi ^{\alpha
_{n}}|P_{x,a}^{\alpha _{n}}\otimes Q_{y,b}^{\alpha _{n}}|\xi ^{\alpha
_{n}}\right\rangle \\
&=&\sum_{n}\left\vert \alpha _{n}\right\vert ^{2}\left\langle \xi ^{\alpha
_{n}}|P_{x,a}^{\alpha _{n}}\Pi _{y,b}^{\alpha _{n}}\otimes I|\xi ^{\alpha
_{n}}\right\rangle \\
&=&\sum_{n}\left\vert \alpha _{n}\right\vert ^{2}\left( \tau ^{\alpha
_{n}}\circ \pi ^{\alpha _{n}}\right) \left( p_{x,a}q_{y,b}\right) \text{.}
\end{eqnarray*}%
By \cite{cw}, we can replace the
infinite convex combination above with a finite convex combination. This
concludes the proof.

(iv)$\Rightarrow $(iii) Suppose that $\cl F$ is a finite-dimensional C*-algebra, $%
\tau $ is a tracial state on $\cl F$, and $\pi :C^{\ast }\left( \mathcal{G}%
\right) \rightarrow F$ is a unital *-homomorphisms such that $p\left(
a,b|x,y\right) =\left( \tau \circ \pi \right) \left( p_{xa}q_{yb}\right) $
for $\left( a,b,x,y\right) \in A\times B\times X\times Y$. After observing
that $\tau $ is a convex combination of canonical tracial states on matrix
algebras, one can proceed as in the proof of (ii)$\Rightarrow $(i) in Theorem %
\ref{Thm:qc} to show that $p\in C_{\mathrm{qm}}$.

(iii)$\Rightarrow $(ii)$\Rightarrow $(i) hold trivially.
\end{proof}

\begin{corollary}
\label{Coro:q}
Let $\mathcal{G}$ be an imitation game. The following
assertions are equivalent:

\begin{itemize}
\item[(i)] $\mathcal{G}$ has a perfect spatial quantum strategy;

\item[(ii)] $\mathcal{G}$ has a perfect quantum strategy;

\item[(iii)] $\mathcal{G}$ has a perfect quantum strategy using a maximally
entangled vector;

\item[(iv)] the C*-algebra of the game $C^{\ast }\left( \mathcal{G}\right) $ is
nonzero, and it has a nondegenerate finite-dimensional representation.
\end{itemize}
\end{corollary}

In view of Proposition \ref{Prop:solution}, Corollary \ref{Coro:qc} recovers 
\cite[Theorem 5]{cleve_perfect_2016}---see also \cite[Section 4]%
{cleve_characterization_2014}---as a particular case.

We also obtain the following result of \cite{kps}.


\begin{corollary}
The set $\cl C_{\mathrm{q}}^{\mathrm{s}} = \cl C_{\mathrm{q}}\cap \mathcal{C}^{\mathrm{s}}$ of
synchronous quantum correlations is equal to the set $\cl C_{\mathrm{qs}}^{%
\mathrm{s}} = \cl C_{\mathrm{qs}}\cap \cl C^{\mathrm{s}}$ of synchronous quantum
spatial correlations, as well as to the convex hull $\cl C_{\mathrm{qm}}^{%
\mathrm{s}}= \cl C_{\mathrm{qm}}\cap \cl C^{\mathrm{s}}$ of the synchronous quantum
correlation defined using a maximally entangled vector.
\end{corollary}

One can similarly characterise the local perfect strategies for 
$\mathcal{G}$ in terms of traces on $C^{\ast }\left( \mathcal{G}\right)$.
We omit the proof since it follows closely the ideas in the proof of Theorems \ref{Thm:qc} and \ref{Thm:q}.


\begin{theorem}\label{th_localcha}
Suppose that $\mathcal{G}$ is an imitation game with input sets $X,Y$ and
output sets $A,B$. Let $p:A\times B\times X\times Y\rightarrow \left[ 0,1%
\right] $ be a correlation. The following statements are equivalent:

\begin{itemize}
\item[(i)] $p\in C_{\mathrm{loc}}\left( \mathcal{G}\right) $;

\item[(ii)] the C*-algebra of the game $C^{\ast }\left( \mathcal{G}\right) $ is
nonzero, and there exists a finite-dimensional abelian C*-algebra $\cl F$
with a tracial state $\tau $ and a unital *-homomorphism $\pi :C^{\ast
}\left( \mathcal{G}\right) \rightarrow \cl F$ such that $p\left( a,b|x,y\right)
=\left( \tau \circ \pi \right) \left( p_{xa}q_{yb}\right) $.
\end{itemize}
\end{theorem}

\begin{corollary} Let $\cl G$ be an imitation game. Then $\cl G$ has a perfect local strategy if and only if there exists a unital *-homomorphism $\pi: C^*(\cl G) \to \bb C$.
\end{corollary}


\section{Mirror games}\label{s_mg}

In this section, we consider the subclass of mirror games and provide a 
different kind  of representation  of their perfect quantum commuting strategies in terms of traces.
The approach is Hilbert-space-free
and allows us to characterise the perfect quantum approximate strategies for 
these games as ones arising from amenable traces, extending significantly the corresponding result for 
synchronous games from \cite{kps}. 

Let $X,Y,A$ and $B$ be finite sets. 
Recall that 
$\cl A(X,A) = \ell^{\infty}(A) \ast_1 \cdots \ast_1 \ell^{\infty}(A)$, where the free product is taken $|X|$ times,
$(e_{x,a})_{a=1}^{|A|}$ is the canonical basis of $x$-th copy of $\ell^{\infty}(A)$, 
and $\sxa = {\rm span}\{e_{x,a} : x\in X, a\in A\}$. 
The canonical generators of $\cl A(Y,B)$ are denoted by $f_{y,b}$, and their span is denoted by $\cl S_{Y,B}$.

Let $\cl G = (X,Y,A,B,\lambda)$ be a non-signalling game. 
For $x\in X$, $y\in Y$, $a\in A$ and $b\in B$, set 
$$E_{x,y}^a = \{b\in B : \lambda(x,y,a,b) = 1\} \ \mbox{ and } \ E_{x,y}^b = \{a\in A : \lambda(x,y,a,b) = 1\}.$$
Recall that $\cl G$ is a \emph{mirror game} if 
there exist functions $\xi : X\to Y$ and $\eta : Y\to X$ such that, of every $x\in X$, we have 
$$E_{x,\xi(x)}^a \cap E_{x,\xi(x)}^{a'} = \emptyset, \ \ a,a'\in A, \ a\neq a',$$
and 
$$E_{\eta(y),y}^b \cap E_{\eta(y),y}^{b'} = \emptyset, \ \ b,b'\in B, \ b\neq b'.$$

In the statement of the following theorem, we use the correspondence $s\to p_s$ between 
states and non-signalling families highlighted in Section \ref{Section:correlations}.

\begin{theorem}\label{th_rest}
Let $\cl G = (X,Y,A,B,\lambda)$ be a mirror game, $p\in \cl C_{\qc}(\lambda)$ and
$s\in S(\cl A(X,A)\otimes_{\max}\cl A(Y,B))$ be such that $p = p_s$. 
Then 

(i) \ \ the functional $\tau : \cl A(X,A)\to \bb{C}$ given by $\tau(z) = s(z\otimes 1)$, $z\in \cl A(X,A)$, 
is a trace;

(ii) \ there exists a unital *-homomorphism $\rho : \cl A(Y,B)\to \cl A(X,A)$ such that
\begin{equation}\label{eq_psrho}
p(a,b|x,y) = \tau(e_{x,a} \rho(f_{y,b})), \ \ \ x\in X, y\in Y, a\in A, b\in B,
\end{equation}
and
\begin{equation}\label{eq_bar}
s(z\otimes f_{y_1,b_1}\cdots f_{y_k,b_k}) = \tau(z \rho(f_{y_k,b_k}\cdots f_{y_1,b_1})), 
\end{equation}
for all $z\in \cl A(A,X)$, $k\in \bb{N}$, $y_i\in Y, b_i\in B$, $i = 1,\dots,k$.
\end{theorem}
\begin{proof}
(i) We assume first that
\begin{equation}\label{eq_fullun}
\cup_{a\in A} E_{x,\xi(x)}^a = B \mbox{ and } \cup_{b\in B} E_{\eta(y),y}^b = A, \ \ x\in X, y\in Y.
\end{equation}
For $x\in X, y\in Y, a\in A$ and $b\in B$, let 
$$p_{x,a} = \sum_{b\in E_{x,\xi(x)}^a} f_{\xi(x),b}, \ \ \ q_{y,b} = \sum_{a\in E_{\eta(y),y}^b} e_{\eta(y),a}.$$
By (\ref{eq_fullun}), $\sum_{b\in B}q_{y,b} = 1$, for all $y\in Y$, and the universal property of 
$\cl A(Y,B)$ implies that the assignment $f_{y,b}\to q_{y,b}$, $y\in Y$, $b\in B$, 
extends to a unital *-homomorphism $\rho : \cl A(Y,B)\to \cl A(X,A)$. 

If $u_1,u_2\in \cl A(X,A)\otimes_{\max}\cl A(Y,B)$, write 
$u_1\sim u_2$ if $s(u_1 - u_2) = 0$. Clearly, $\sim$ is an equivalence relation. 
Fix $x\in X$ and $a\in A$. Then
$$s(e_{x,a}\otimes 1) = \sum_{b\in B} s(e_{x,a}\otimes f_{\xi(x),b}) = \sum_{b\in E_{x,\xi(x)}^a} s(e_{x,a}\otimes f_{\xi(x),b}) 
= s(e_{x,a}\otimes p_{x,a}).$$
On the other hand, if $a'\neq a$ then 
$E_{x,\xi(x)}^{a'} \cap E_{x,\xi(x)}^{a} = \emptyset$ and so 
$s(e_{x,a'}\otimes f_{\xi(x),b}) = 0$ whenever $b\in E_{x,\xi(x)}^a$, implying
$$s(e_{x,a'}\otimes p_{x,a}) = \sum_{b\in E_{x,\xi(x)}^a} s(e_{x,a'}\otimes f_{\xi(x),b}) = 0.$$
It follows that 
$$s(1\otimes p_{x,a}) = \sum_{a'\in A} s(e_{x,a'}\otimes p_{x,a}) = s(e_{x,a}\otimes p_{x,a}).$$
Thus,
$$e_{x,a}\otimes 1 \sim  e_{x,a}\otimes p_{x,a} \sim 1\otimes p_{x,a}, \ \ \ x\in X, a\in A.$$

Write $h_{x,a} = e_{x,a}\otimes 1  - 1\otimes p_{x,a}$. 
Clearly, $h_{x,a}$ is selfadjoint and 
$$h_{x,a}^2 = e_{x,a}\otimes 1  - e_{x,a}\otimes p_{x,a} - e_{x,a}\otimes p_{x,a} + 1\otimes p_{x,a};$$
thus, $h_{x,a}^2 \sim 0$. 
The Cauchy-Schwarz inequality now implies 
\begin{equation}\label{eq_uhe}
u h_{x,a} \sim 0 \mbox{ and } h_{x,a} u \sim 0, \ \ x\in X, a\in A, \ u\in \cl A(X,A)\otimes_{\max}\cl A(Y,B).
\end{equation}
In particular, 
\begin{equation}\label{eq_zeez}
ze_{x,a}\otimes 1 \sim z\otimes p_{x,a} \sim e_{x,a} z\otimes 1, \ \ x\in X, a\in A, \ z\in \cl A(X,A).
\end{equation}
Similarly, setting $h_{y,b} = q_{y,b}\otimes 1  - 1\otimes f_{y,b}$, where $y\in Y$ and $b\in B$, 
we obtain $h_{y,b}^2 \sim 0$, and therefore
\begin{equation}\label{eq_zeez2}
zq_{y,b}\otimes 1 \sim z\otimes f_{y,b} \sim q_{y,b} z\otimes 1, \ \ y\in Y, b\in B, \ z\in \cl A(X,A),
\end{equation}
and
\begin{equation}\label{eq_zeez2}
zq_{y,b}\otimes w \sim z\otimes w f_{y,b}, \ \ y\in Y, b\in B, \ z\in \cl A(X,A), w\in \cl A(Y,B).
\end{equation}

It is clear that $\tau$ is a state on $\cl A(X,A)$. 
Let $z$ and $w$ be words on the set $\cl E := \{e_{x,a} : x\in X, a\in A\}$. 
We show by induction on the length $|w|$ of $w$ that 
\begin{equation}\label{eq_zwwz}
zw\otimes 1 \sim wz\otimes 1. 
\end{equation}
In the case $|w| = 1$, the claim reduces to (\ref{eq_zeez}). Suppose (\ref{eq_zwwz}) holds if $|w|\leq n-1$.
Let $|w| = n$ and write $w = w'e$, where $e\in \cl E$. Then, using (\ref{eq_zeez}), we have 
$$zw \otimes 1 = zw'e \otimes 1 \sim ezw' \otimes 1 \sim w'ez \otimes 1 = wz \otimes 1.$$
From (\ref{eq_zwwz}) and the fact that 
the set of all linear combinations of words on $\cl E$ is dense in $\cl A$, we conclude that $\tau$ is a trace on $\cl A(X,A)$. 

Now assume that the conditions from (\ref{eq_fullun}) are not fulfilled. Choose $a_0\in A$ and $b_0\in B$, 
and define $p_{x,a}$, $a\neq a_0$, $q_{y,b}$, $b\neq b_0$,
as in (\ref{eq_fullun}).  
Set $E_{y,0} = A\setminus \left(\cup_{b\in B} E_{\eta(y),y}^b\right)$, 
$F_{x,0} = B\setminus\left(\cup_{a\in A} E_{x,\xi(x)}^a\right)$, and let 
$$p_{x,a_0} = \sum_{b\in F_{x,0}} f_{\xi(x),b}, \ \ \ q_{y,b_0} = \sum_{a\in E_{y,0}} e_{\eta(y),a}.$$
The proof thereafter proceeds as before.

(ii) By (\ref{eq_zeez2}), 
$$p(a,b|x,y) = s(e_{x,a}\otimes f_{y,b}) = s(e_{x,a}q_{y,b}\otimes 1) = \tau(e_{x,a}q_{y,b}) = \tau(e_{x,a} \rho(f_{y,b})),$$
for all $x\in X, y\in Y, a\in A, b\in B$.
We show (\ref{eq_bar}) by induction on $k$. If $k = 1$, the claim follows from (\ref{eq_zeez2}). 
Assuming validity for up to $k-1$ terms, using (\ref{eq_zeez2}) we have 
\begin{eqnarray*}
s(z\otimes f_{y_1,b_1}\cdots f_{y_k,b_k}) 
& = & 
s(z q_{y_k,b_k}\otimes f_{y_1,b_1}\cdots f_{y_{k-1},b_{k-1}})\\
& = & 
\tau(z \rho(f_{y_k,b_k}) \rho(f_{y_{k-1},b_{k-1}}\cdots f_{y_{1},b_{1}}))\\
& = & 
\tau(z \rho(f_{y_k,b_k} f_{y_{k-1},b_{k-1}}\cdots f_{y_{1},b_{1}})).
\end{eqnarray*}
\end{proof}

\noindent {\bf Remark. } Theorem \ref{th_rest} is a generalisation 
of \cite[Theorem 5.5]{paulsen_estimating_2016}, which was concerned with synchronous games, 
in two directions: it specifies concretely the
C*-algebra involved in the strategy representation, and it extends this representation to the 
wider class of mirror games.

\medskip

Recall that, if $\cl A$ is a C*-algebra, its opposite C*-algebra $\cl A^{\rm op}$ is defined to be the same 
involutive normed space, whose elements are denoted by $a^{\rm op}$, $a\in \cl A$, 
but equipped with the product given by $a^{\rm op}b^{\rm op} = (ba)^{\rm op}$. 
The following lemma was established in \cite{kps}.

\begin{lemma}\label{l_op}
There exists a *-isomorphism $\gamma : \cl A(X,A)\to \cl A(X,A)^{\rm op}$ such that 
$$\gamma(e_{x_1,a_1}\cdots e_{x_k,a_k}) = (e_{x_k,a_k}\cdots e_{x_1,a_1})^{\rm op}, 
\ \ x_i\in X, a_i\in A, i = 1,\dots,k, k\in \bb{N}.$$
\end{lemma}

\begin{theorem}\label{th_qcamen}
Let $\cl G = (X,Y,A,B,\lambda)$ be a mirror game and $p\in \cl C_{\rm ns}(\lambda)$.
The following are equivalent: 

(i) \ $p\in \cl C_{\rm qc}(\lambda)$; 

(ii) there exist a trace $\tau : \cl A(X,A)\to \bb{C}$ and a unital *-homomorphism $\rho : \cl A(Y,B)\to \cl A(X,A)$ with 
$\rho(\cl S_{X,A})\subseteq \cl S_{Y,B}$ such that 
$$p(a,b|x,y) = \tau(e_{x,a} \rho(f_{y,b})), \ \ \ x\in X, y\in Y, a\in A, b\in B.$$
\end{theorem}
\begin{proof}
(i)$\Rightarrow$(ii) follows from Theorem \ref{th_rest} and its proof, using the fact that 
the projections $q_{y,b}$ defined therein lie in $\cl S_{X,A}$.

(ii)$\Rightarrow$(i) Let $\phi : \cl A(X,A)\otimes \cl A(X,A)^{\rm op} \to \bb{C}$ be the bilinear form given by 
$$\phi(z\otimes w^{\rm op}) = \tau(zw), \ \ \ z,w \in \cl A(X,A).$$
It is well-known that $\phi$ extends to a state on $\cl A(X,A)\otimes_{\max} \cl A(X,A)^{\rm op}$
(this can be seen, for example, by noting that $\phi$ is unital and jointly completely positive, and using 
results from \cite{kptt_tensor}). 
Let $s : \cl A(X,A)\otimes_{\max}\cl A(Y,B)\to \bb{C}$ be the linear functional defined by 
$$s = \phi \circ ({\rm id}\otimes\gamma)\circ ({\rm id} \otimes \rho).$$
Then $s$ is a state on $\cl A(X,A)\otimes_{\max}\cl A(Y,B)$, and 
$$s(e_{x,a}\otimes f_{y,b})
= \phi(e_{x,a} \otimes \gamma(\rho(f_{y,b})))
= \phi(e_{x,a} \otimes \rho(f_{y,b})^{\rm op}) = \tau(e_{x,a}\rho(f_{y,b})).$$
By Corollary \ref{th_qcgen}, $p\in \cl C_{\rm qc}(\lambda)$.
\end{proof}

For the next theorem, we will need the notion of an amenable trace; we refer the reader to 
\cite{brown_c*-algebras_2008}) for its many equivalent definitions and properties. The definition that we shall use is that a trace $\tau$ on a C*-algebra $\cl A$  is {\it amenable} if and only if the map on the algebraic tensor product
$\phi: \cl A \otimes \cl A^{op} \to \bb C$ given by $\phi(a \otimes b^{op}) = \tau(ab)$ is bounded with respect to the minimal C*-tensor norm.


\begin{theorem}\label{th_qaamen}
Let $\cl G = (X,Y,A,B,\lambda)$ be a mirror game and $p\in \cl C_{\rm ns}(\lambda)$.
The following are equivalent: 

(i) \ $p\in \cl C_{\rm qa}(\lambda)$; 

(ii) there exist an amenable trace $\tau : \cl A(X,A)\to \bb{C}$ and a unital *-homomorphism $\rho : \cl A(Y,B)\to \cl A(X,A)$ with 
$\rho(\cl S_{X,A})\subseteq \cl S_{Y,B}$ such that 
$$p(a,b|x,y) = \tau(e_{x,a} \rho(f_{y,b})), \ \ \ x\in X, y\in Y, a\in A, b\in B.$$
\end{theorem}
\begin{proof}
(ii)$\Rightarrow$(i) 
Let $\phi : \cl A(X,A)\otimes \cl A(X,A)^{\rm op} \to \bb{C}$ be the bilinear form given by 
$$\phi(z\otimes w^{\rm op}) = \tau(zw), \ \ \ z,w \in \cl A(X,A).$$
Since $\tau$ is amenable, $\phi$ extends to a state on $\cl A(X,A)\otimes_{\min} \cl A(X,A)^{\rm op}$ 
\cite[Theorem 6.2.7]{brown_c*-algebras_2008}. 
Define a state $s$ on $\cl A(X,A)\otimes_{\min}\cl A(Y,B)\to \bb{C}$ by letting
$$s = \phi \circ ({\rm id}\otimes\gamma)\circ ({\rm id} \otimes \rho).$$
As in the proof of Theorem \ref{th_qcamen}, 
$$p(a,b|x,y) = s(e_{x,a}\otimes f_{y,b}), \ \ \ x\in X, y\in Y, a\in A, b\in B;$$
by Corollary \ref{th_qagen}, $p\in \cl C_{\rm qa}(\lambda)$.

(i)$\Rightarrow$(ii) 
Since $p\in \cl C_{\rm qa}(\lambda)$, by Corollary \ref{th_qagen}, there 
exists a state $s$ on $\cl A(X,A)\otimes_{\min} \cl A(Y,B)$ such that 
$$p(a,b|x,y) = s(e_{x,a}\otimes f_{y,b}), \ \ \ x\in X, y\in Y, a\in A, b\in B.$$
Arguing as in the proof of Theorem \ref{th_rest} and using (\ref{eq_zeez}), 
we may obtain a unital *-homomorphism $\rho : \cl A(Y,B)\to \cl A(X,A)$
with $\rho(\cl S_{Y,B})\subseteq \cl S_{X,A}$ and a unital *-homomorphism 
$\pi : \cl A(X,A)$ $\to$ $\cl A(Y,B)$ with $\pi(\cl S_{X,A})\subseteq \cl S_{Y,B}$ 
such that:
\begin{itemize}
\item[(a)] the functional $\tau : \cl A(X,A)\to \bb{C}$, given by $\tau(z) = s(z\otimes 1)$, is a trace;
\item[(b)] $s(z\otimes \pi(e_{x_1,a_1}\cdots e_{x_k,a_k})) = \tau(z e_{x_k,a_k}\cdots e_{x_1,a_1})$, for $z\in \cl A(A,X), x_i\in X, a_i\in A$, 
$i = 1,\dots, k$, and
\item[(c)] $p(a,b|x,y) = \tau(e_{x,a}\rho(f_{y,b}))$, for $x\in X$, $y\in Y$, $a\in A$, $b\in B$. 
\end{itemize}

Let $\phi : \cl A(X,A)\otimes_{\min} \cl A(X,A)^{\rm op} \to \bb{C}$ be the state defined by letting 
$$\phi = s\circ ({\rm id}\otimes \pi) \circ ({\rm id}\otimes \gamma^{-1}).$$
Let $z\in \cl A(X,A)$ and $w = e_{x_1,a_1}\cdots e_{x_k,a_k}$, for some $x_i\in X$, $a_i\in A$, $i = 1,\dots,k$. 
Set $\bar{w}:= \gamma^{-1}(w^{\rm op}) = e_{x_k,a_k}\cdots e_{x_1,a_1}$. 
Thus, using (b) we have 
\begin{eqnarray*}
\phi(z\otimes w^{\rm op}) 
& = & 
s(z\otimes \pi(\bar{w})) = s(z\otimes \pi(e_{x_k,a_k}\cdots e_{x_1,a_1}))\\
& = & 
\tau(z e_{x_1,a_1}\cdots e_{x_k,a_k}) = \tau(zw).
\end{eqnarray*}
By linearity and continuity, 
$$\phi(z\otimes w^{\rm op}) = \tau(zw), \ \ \ \ z, w\in \cl A(X,A).$$
By \cite[Theorem 6.2.7]{brown_c*-algebras_2008}, $\tau$ is amenable. 
\end{proof}

\medskip

\noindent {\bf Acknowledgement.} 
Part of this research was conducted during two Focused Research Meetings, funded by the Heilbronn Institute, and hosted at Queen's
University Belfast in October 2016 and March 2017. 
G. Scarpa acknowledges the support of
MTM2014-54240-P (MINECO), QUITEMAD+-CM Reference: S2013/ICE-2801 (Comunidad de Madrid), ICMAT Severo Ochoa project SEV-2015-0554 (MIN-ECO), and grant 48322 from the John Templeton Foundation. 
The opinions expressed in this publication are those of the authors and do not necessarily reflect the 
views of the John Templeton Foundation.


\begin{thebibliography}{99}

\bibitem{qiso}
{\sc A. Atserias, L. Man\v{c}inska, D. E. Roberson, R. \v{S}\'{a}mal, S. Severini and A. Varvitsiotis},
{\it Quantum and non-signalling graph isomorphisms},
{\rm preprint (2016), arXiv:1611.09837v3}.

\bibitem{b}
{\sc J. Barrett},
{\it Information processing in generalized probabilistic theories},
{\rm Phys. Rev. A 75 2007, 032304}. 

\bibitem{brown_c*-algebras_2008}
{\sc N. P. Brown and N. Ozawa},
{\it C*-algebras and finite-dimensional approximations},
{\rm American Mathematical Society, 2008}.



\bibitem{cameron_quantum_2007}
{\sc  P. J. Cameron, A. Montanaro, M. W. Newman, S. Severini and A. Winter},
{\it On the quantum chromatic number of a graph},
{\rm Electronic J. Combinatorics 14 (2007), no. 1,  Paper 81, 15 pp.}


\bibitem{CE2} 
{\sc M.D. Choi and E.G. Effros}, 
{\it Injectivity and operator spaces,} 
{\rm J. Funct. Anal. 24 (1977), 156-209}.


\bibitem{cleve_perfect_2016}
{\sc  R. Cleve, L. Liu and W. Slofstra},
{\it Perfect commuting-operator strategies for linear system games},
{\rm preprint (2016), arXiv:1606.02278}.



\bibitem{cleve_characterization_2014}
{\sc R. Cleve and R. Mittal},
{\it Characterization of binary constraint system games},
{\rm Automata, languages, and programming. Part I, Lecture {Notes} in {Comput}. {Sci}., vol. 8572, Springer (2014), 320-331}.


\bibitem{c}
{\sc A. Connes}, 
{\it Classification of injective factors}, 
{\rm Ann. of Math. (2) 104 (1976), 73-115}.


\bibitem{cw}
{\sc W. D. Cook and R. J. Webster},
{\it Carath\'eodory's theorem},
{\rm Canad. Math. Bull. 15 (1972), 293-293}.


\bibitem{dpp_delta}
{\sc K. Dykema, V. I. Paulsen, J. Prakash},
{\it The Delta game},
{\rm preprint (2017), arXiv:1707.06186}.


\bibitem{dpp}
{\sc K. Dykema, V. I. Paulsen, J. Prakash},
{\it Non-closure of the set of quantum correlations via graphs},
{\rm preprint (2017), arXiv:1709.05032}.


\bibitem{effros_operator_2000}
{\sc E. G. Effros and Zh.-J. Ruan},
{\it Operator spaces},
{\rm Oxford University Press, 2000}.


\bibitem{fkpt}
{\sc D. Farenick, A. Kavruk, V. I. Paulsen and I. G. Todorov},
{\it Operator systems from discrete groups}, 
{\rm Comm. Math. Phys. 329 (2014), 207-238}.

\bibitem{fkpt_NYJ}
{\sc D. Farenick, A. Kavruk, V. I. Paulsen and I. G. Todorov},
{\it Characterizations of the weak expectation property}, 
{\rm New York J. Math., to appear}.


\bibitem{fp} 
{\sc D. Farenick and V. I. Paulsen},
{\it Operator system quotients of matrix algebras and their tensor products},
{\rm Math. Scand. 111 (2012), 210-243}.



\bibitem{fritz_operator_2014}
{\sc T. Fritz},
{\it Operator system structures on the unital direct sum of C*-algebras},
{\rm Rocky Mountain J. Math. 44 (2014), no. 3, 913-936}.


\bibitem{helton_algebras_2017}
{\sc  J. W. Helton, K. P. Meyer, V. I. Paulsen and M. Satriano},
{\it Algebras, synchronous games, and chromatic numbers of graphs},
{\rm preprint (2017), arXiv:1703.00960}.


\bibitem{holevo}
{\sc A. S. Holevo},
{\it Quantum systems, channels, information},
{\rm De Gruyter, 2013}.

\bibitem{jnppsw}
{\sc M. Junge, M. Navascues, C. Palazuelos, D. Perez-Garcia, V. B. Scholtz and R. F. Werner},
{\it Connes' embedding problem and Tsirelson's problem},
{\rm J. Math. Physics 52 (2011), 012102}.




\bibitem{kavruk}
{\sc A. S. Kavruk},
{\it Nuclearity related properties in operator systems},
{\rm J. Operator Theory 71 (2014), no. 1, 95-156}.


\bibitem{kptt_tensor}
{\sc A. Kavruk, V. I. Paulsen, I. G. Todorov and M. Tomforde},
{\it Tensor products of operator systems},
{\rm J. Funct. Anal. 261 (2011), no. 2, 267-299}.


\bibitem{kptt2010}
{\sc A. S. Kavruk, V. I. Paulsen, I. G. Todorov, and M. Tomforde},
{\it Quotients, exactness, and nuclearity in the operator system category},
{\rm Adv. Math. 235 (2013), 321-360}.



\bibitem{kps}
{\sc S.-J. Kim, V. I. Paulsen and C. Schafhauser},
{\it A synchronous game for binary constraint systems}, 
{\rm preprint (2017), arXiv:1707.01016}.

\bibitem{qhoms}
{\sc L. Man\v{c}inska and D. E. Roberson},
{\it Quantum homomorphisms},
{\rm J. Combin. Theory Ser. B, 118 (2016), 228-267}.


\bibitem{mv}
{\sc L. Man\v{c}inska and T. Vidick},
{\it Unbounded entanglement in nonlocal games}, 
{\rm Quantum Inf. Comput. 15 (2015), no. 15-16, 1317-1332}.


\bibitem{Rthesis}
{\sc D. E. Roberson},
{\it Variations on a theme: Graph homomorphisms},
{\rm Ph.D.~thesis, University of Waterloo, 2013.}.

\bibitem{ortiz_quantum_2016}
{\sc C. M. Ortiz and V. I. Paulsen},
{\it Quantum graph homomorphisms via operator systems},
{\rm Linear Algebra App. 497 (2016), 23-43}.


\bibitem{oz} 
{\sc N. Ozawa}, 
{\it About the Connes' embedding problem--algebraic approaches},
{\rm Japan. J.  Math. 8 (2013), no. 1, 147-183}.


\bibitem{pv} 
{\sc C. Palazuelos and T. Vidick}, 
{\it Survey on nonlocal games and operator space theory}, 
{\rm J. Math. Phys. 57 (2016), no. 1, 015220}.


\bibitem{paulsen_completely_2002}
{\sc V. I. Paulsen},
{\it Completely bounded maps and operator algebras},
{\rm Cambridge University Press, 2002}.


\bibitem{paulsen_estimating_2016}
{\sc V. I. Paulsen, S. Severini, D. Stahlke, I. G. Todorov and A. Winter},
{\it Estimating quantum chromatic numbers},
{\rm J. Funct. Anal. 270 (2016), no. 6, 2188-2222}.


\bibitem{paulsen_quantum_2015} 
{\sc  V. I. Paulsen and I. G. Todorov},
{\it Quantum chromatic numbers via operator systems},
{\rm Q. J. Math. 66 (2015), no. 2, 677-692}.



\bibitem{pisier_introduction_2003}
{\sc  G. Pisier},
{\it Introduction to operator space theory},
{\rm Cambridge University Press, 2003}.


\bibitem{RSU}
{\sc M. V. Ramana, E. R. Scheinerman and D. Ullman},
{\it Fractional isomorphism of graphs},
{\rm Discrete Math. 132 (1994), no. 1-3, 247-265}.


\bibitem{rao}
{\sc A. Rao},
{\it Parallel repetition in projection games and a concentration bound},
{\rm SIAM J. Comput. 40 (2011), no. 6, 1871-1891}.



\bibitem{s_2016}
{\sc W. Slofstra},
{\it Tsirelson's problem and an embedding theorem for groups arising from non-local games},
{\rm preprint (2016), arXiv:1606.03140}.


\bibitem{s_2017}
{\sc W. Slofstra},
{\it The set of quantum correlations is not closed},
{\rm preprint (2017), arXiv:1703.08618}.


\bibitem{sv}
{\sc W. Slofstra and T. Vidick},
{\it Entanglement in non-local games and the hyperlinear profile of groups},
{\rm preprint (2017), arXiv:1711.10676}.
	

\bibitem{tsirelson1980}
{\sc B. S. Tsirelson},
{\it Quantum generalizations of Bell's inequality},
{\rm Lett. Math. Phys. 4 (1980), no. 4, 93-100}.

\bibitem{tsirelson1993}
{\sc B. S. Tsirelson},
{\it Some results and problems on quantum {B}ell-type inequalities},
{\rm Hadronic J. Suppl. 8 (1993), no. 4, 329-345}.



\end{thebibliography}
\end{document}